\documentclass[a4paper,12pt,oneside]{article}

\usepackage[latin1]{inputenc}
\usepackage{eucal}
\usepackage{exscale}
\usepackage{latexsym}
\usepackage[T1]{fontenc}
\usepackage{amsmath}
\usepackage{graphicx}
\usepackage{amsfonts,makeidx,color,amsmath, amssymb,graphics,mathrsfs,eurosym,amsthm}
\usepackage{epsfig}

\newtheorem{theorem}{Theorem}[section]
\newtheorem{lemma}{Lemma}[section]
\theoremstyle{remark}
\newtheorem{remark}{Remark}[section]

\font\fivrm=cmr5 \relax
\input prepictex
\input pictex
\input postpictex

\def\RR{\mathbb{R}}

\title{\bf Curve network interpolation by $C^1$ quadratic B-spline surfaces}
\author{C. Dagnino, P. Lamberti and S. Remogna\thanks{Department of Mathematics, University of Torino, via C. Alberto 10, 10123 Torino, Italy ({\tt catterina.dagnino@unito.it, paola.lamberti@unito.it, sara.remogna@unito.it})}}

\date{}
\begin{document}

\maketitle

\begin{abstract} In this paper we investigate the problem of interpolating a B-spline curve network, in order to create a surface satisfying such a constraint and defined by blending functions spanning the space of bivariate $C^1$ quadratic splines on criss-cross triangulations. We prove the existence and uniqueness of the surface, providing a constructive algorithm for its generation. We also present numerical and graphical results and comparisons with other methods. 

\end{abstract}\vskip 10pt


{\sl Subject classification AMS (MOS)}: 65D05, 65D07, 65D17

\section{Introduction}
The interpolation of a curve network by a smooth parametric surface has been a well-studied problem in literature since the second half of last century. Indeed it is easier and more intuitive to describe a complicated 3D free-form surface by creating a curve network than to directly manage surface control points. 

Given a curve network, in the literature this interpolation problem is faced up by considering different approaches. For instance a classical one is based on blending-function methods \cite{Bar,Bos,del,gor1,gor2,gor3}. Some other ones consist in the interpolation of the curve mesh either by a smooth regularly parametrized surface with one polynomial piece per facet \cite{pet} or by subdivision schemes \cite{con}.

We recall that, in general, most surface fitting methods fall into one of two categories: global or local methods.
A global method represents a surface by a ready available expression on the whole parametric domain and it is generally obtained either by solving  a system of equations or by other schemes, for instance the ones based on the quasi-interpolation, that do not require to solve any system of equations.
Local methods are more geometric in nature, constructing the surface patch-wise, using only local data for each step and imposing the required smoothness patch by patch. 

Moreover it is well known that NURBS, based on rectangular patches, are a widely used computational geometry technology in CAGD, thanks to their many good properties.
Tensor product B-spline surfaces are a specific kind of NURBS \cite{hl,pie}.
However, sometimes the above surfaces can create unwanted oscillations and some inflection points on the surface, due to their high coordinate degree.
An example of tensor product B-spline surface interpolating a B-spline curve network is proposed in \cite[Chap. 10]{pie}.
 It is obtained from a blending-function method based on a special application of Gordon type  schemes \cite{gor1,gor2,gor3}, constructing three different surfaces expressed by means of B-spline functions belonging to different spaces, so that the bivariate resulting one is not defined on the rectangular grid associated to the knot vectors of the curve net, since a knot refinement is required. 

NURBS on triangulations can avoid the unwanted oscillations, since they have a  total degree lower than the one of NURBS based on rectangular patches. Therefore they can be very useful in CAGD \cite{LW2,LW3,LW4}.

In this paper we want to combine the advantageous features of the above cited techniques, proposing a new global  method, to generate a surface interpolating a given B-spline curve network and using as blending functions  bivariate B-splines  on a criss-cross triangulation $T_{mn}$ of the parametric domain. In particular, we consider  $C^1$ quadratic B-spline curve networks and we define a parametric surface that satisfies the above interpolation  constraint and it is based on the B-spline functions spanning the space $S_2^1(T_{mn})$ of all quadratic $C^1$ splines on  $T_{mn}$.

We remark that this space has been widely studied, with reference to its dimension, local basis, approximation power, etc. and it has been used in many applications  (see \cite{DL1,DL5,DL4,DS,L,S1,S3,W,LW2} and references therein). This paper wants also to be a further contribution to the researches on the surface construction based on above blending functions spanning $S_2^1(T_{mn})$.

The structure of this paper is the following.

In Section \ref{Sect2} we define the B-spline curve network with its compatibility conditions.

In Section \ref{sec_int} we give necessary and sufficient conditions for the existence of the B-spline interpolating surface and we prove its uniqueness. Moreover, we propose a constructive algorithm for its generation and some notes on the behaviour of the surface control points generation process with respect to the round-off error.

Finally, in Section \ref{res_num} we present some numerical and graphical applications. In particular we numerically validate the uniqueness of the surface interpolating the given B-spline curve network and we compare our method with two other ones, based on classical biquadratic tensor product B-spline functions. Moreover we underline that  the rectangular topology of the parametric domain is not a too tight constraint for the object shape.


\section{The B-spline curve network}\label{Sect2}
\noindent
Let $\{P_j^{(r)}\}_{j=0}^{n+1}$, $r=0,\dots,m$ and $\{Q_i^{(s)}\}_{i=0}^{m+1}$, $s=0,\dots,n$, $P_j^{(r)}$, $Q_i^{(s)} \in \RR^3$, be the control points of two given sets of $C^1$ quadratic B-spline curves
\begin{equation}\label{rete_curve}
\begin{array}{l} 
\phi_r(v)=\displaystyle\sum_{j=0}^{n+1}P_j^{(r)}B_j(v|V),\quad\quad r=0,\dots,m, \quad\quad v\in [c,d] \subset \RR,\\
\psi_s(u)=\displaystyle\sum_{i=0}^{m+1}Q_i^{(s)}B_i(u|U),\quad\quad s=0,\dots,n, \quad\quad u\in [a,b] \subset \RR,
\end{array}
\end{equation}
where we assume that the curves (\ref{rete_curve}) satisfy the following compatibility conditions:
\begin{itemize} 
\item[C1.] as independent sets, they are compatible in the B-spline sense, that is, all the $\psi_s(u)$ are defined on a common parametric knot vector
\begin{equation} 
U=\{a= u_{-2}\equiv u_{-1}\equiv u_0 < u_1 < \cdots <  u_{m}\equiv u_{m+1}\equiv u_{m+2}=b\} \label{vetU}
\end{equation}
and all the $\phi_r(v)$ are defined on a common parametric knot vector
\begin{equation} 
V=\{c=v_{-2}\equiv v_{-1}\equiv v_0  < v_1 < \cdots <   v_{n}\equiv v_{n+1}\equiv v_{n+2}=d\}.\label{vetV}
\end{equation}
In our notation the B-splines $B_i(u|U)$ and $B_j(v|V)$ have supports $[u_{i-2},u_{i+1}]$ and $[v_{j-2},v_{j+1}]$, respectively \cite{S1,S3}. 

Moreover, we assume $(\xi_i,v_s)$, with 
\begin{equation}\label{si}
	\xi_i=\frac{u_{i-1}+u_i}{2},\ i=0,\ldots,m+1
\end{equation}
and $(u_r,\eta_j)$, with 
\begin{equation}\label{tj}
\eta_j=\frac{v_{j-1}+v_j}{2},\ j=0,\ldots,n+1
\end{equation}
as the pre-image of $Q_i^{(s)}$ for all $s$ and $P_j^{(r)}$ for all $r$, respectively (Fig. \ref{par_domain}). We recall that $\xi_i$ and $\eta_j$ are known as Greville abscissae;
\item[]
\item[C2.] \begin{equation} \label{incroci} I_{rs}=\phi_r(v_s)=\psi_s(u_r), \ \ \ r=0,\ldots,m,\ \ s=0,\ldots,n. \end{equation}
\end{itemize}

\begin{figure}[ht!]
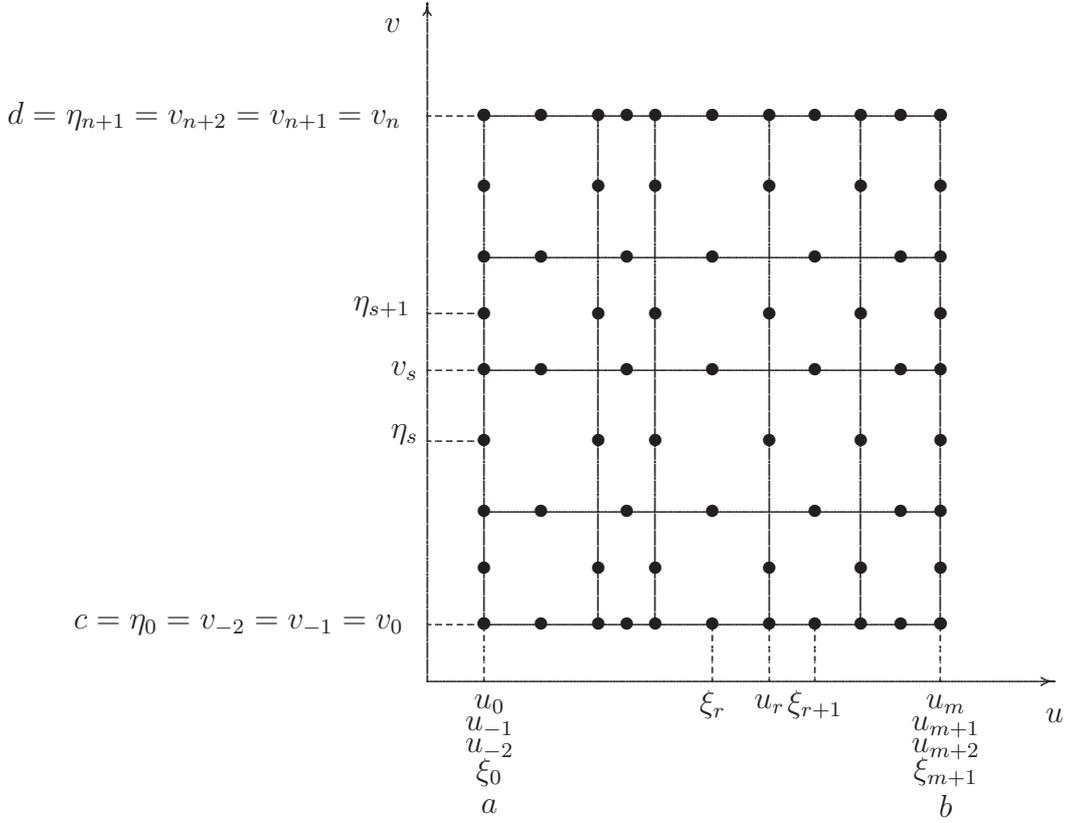

$$
\beginpicture
\setcoordinatesystem units <1.5cm,1.5cm>
\setplotarea x from -1 to 4, y from -1 to 5.5
\setlinear
\plot -0.5 -0.5  4 -0.5 /
\plot -0.5 -0.5  -0.5 4.5 /
\plot 0 0  4 0 /
\plot 0 0  0 4.5 /
\plot  4 4.5  0 4.5 /
\plot  4 4.5 4 0 /
\plot 1 0  1  4.5 /
\plot 1.5 0  1.5  4.5 /
\plot 2.5 0  2.5  4.5 /
\plot 3.3 0  3.3  4.5 /
\plot  0 1  4  1 /
\plot  0 2.25  4  2.25 /
\plot  0 3.25  4  3.25 /
\arrow <5pt> [.1,.6] from -.5 4.5 to -.5 5.5
\arrow <5pt> [.1,.6] from 4 -.5 to  5 -.5
\put {$v$} at -.8  5.3
\put {$u$} at 5  -.8
\setdashes
\setdashpattern <2pt,2pt>
\plot 0 0   0 -0.5 /
\plot 4 0   4 -0.5 /
\plot 0 0   -0.5 0 /
\plot 0 4.5  -0.5 4.5 /
\plot 2 0   2 -0.5 /
\plot 2.5 0  2.5  -0.5 /
\plot 2.9 0   2.9 -0.5 /
\plot 0 1.625  -0.5 1.625 /
\plot 0 2.25  -0.5 2.25 /
\plot 0 2.75  -0.5 2.75 /
\put {$u_{r}$} at 2.5 -.7
\put {$v_{s}$} at -.7 2.25
\put {$\xi_{r}$} at 2.0 -.7
\put {$\xi_{r+1}$} at 2.9 -.7
\put {$\eta_{s}$} at -.7 1.6875
\put {$\eta_{s+1}$} at -.9 2.8125
\put { $a$} at 0 -1.6
\put { $\xi_0$} at 0 -1.3
\put { $u_0$} at 0 -0.7
\put { $u_{-1}$} at 0 -.9
\put { $u_{-2}$} at 0 -1.1
\put { $u_{m}$} at 4 -.7
\put { $u_{m+1}$} at 4 -.9
\put { $u_{m+2}$} at 4 -1.1
\put { $\xi_{m+1}$} at 4 -1.3
\put { $b$} at 4 -1.6
\put { $c=\eta_0=v_{-2}=v_{-1}=v_0$} at -2.2 0
\put { $d=\eta_{n+1}=v_{n+2}=v_{n+1}=v_{n}$} at -2.5 4.5
\put {$\bullet$} at 0 0
\put {$\bullet$} at 0.5 0
\put {$\bullet$} at 1.25 0
\put {$\bullet$} at  2 0
\put {$\bullet$} at  2.9 0
\put {$\bullet$} at  3.65 0
\put {$\bullet$} at  4 0
\put {$\bullet$} at 0 1
\put {$\bullet$} at 0.5 1
\put {$\bullet$} at 1.25 1
\put {$\bullet$} at  2 1
\put {$\bullet$} at  2.9 1
\put {$\bullet$} at  3.65 1
\put {$\bullet$} at  4 1
\put {$\bullet$} at 0 2.25
\put {$\bullet$} at 0.5 2.25
\put {$\bullet$} at 1.25 2.25
\put {$\bullet$} at  2 2.25
\put {$\bullet$} at  2.9 2.25
\put {$\bullet$} at  3.65 2.25
\put {$\bullet$} at  4 2.25
\put {$\bullet$} at 0 3.25
\put {$\bullet$} at 0.5 3.25
\put {$\bullet$} at 1.25 3.25
\put {$\bullet$} at  2 3.25
\put {$\bullet$} at  2.9 3.25
\put {$\bullet$} at  3.65 3.25
\put {$\bullet$} at  4 3.25
\put {$\bullet$} at 0 4.5
\put {$\bullet$} at 0.5 4.5
\put {$\bullet$} at 1.25 4.5
\put {$\bullet$} at  2 4.5
\put {$\bullet$} at  2.9 4.5
\put {$\bullet$} at  3.65 4.5
\put {$\bullet$} at  4 4.5
\put {$\bullet$} at  0 0
\put {$\bullet$} at  0 0.5
\put {$\bullet$} at  0 1.625
\put {$\bullet$} at  0 2.75
\put {$\bullet$} at  0 3.875
\put {$\bullet$} at  0 4.5
\put {$\bullet$} at  1 0
\put {$\bullet$} at  1 0.5
\put {$\bullet$} at  1 1.625
\put {$\bullet$} at  1 2.75
\put {$\bullet$} at  1 3.875
\put {$\bullet$} at  1 4.5
\put {$\bullet$} at  1.5 0
\put {$\bullet$} at  1.5 0.5
\put {$\bullet$} at  1.5 1.625
\put {$\bullet$} at  1.5 2.75
\put {$\bullet$} at  1.5 3.875
\put {$\bullet$} at  1.5 4.5
\put {$\bullet$} at  2.5 0
\put {$\bullet$} at  2.5 0.5
\put {$\bullet$} at  2.5 1.625
\put {$\bullet$} at  2.5 2.75
\put {$\bullet$} at  2.5 3.875
\put {$\bullet$} at  2.5 4.5
\put {$\bullet$} at  3.3 0
\put {$\bullet$} at  3.3 0.5
\put {$\bullet$} at  3.3 1.625
\put {$\bullet$} at  3.3 2.75
\put {$\bullet$} at  3.3 3.875
\put {$\bullet$} at  3.3 4.5
\put {$\bullet$} at  4 0
\put {$\bullet$} at  4 0.5
\put {$\bullet$} at  4 1.625
\put {$\bullet$} at  4 2.75
\put {$\bullet$} at  4 3.875
\put {$\bullet$} at  4 4.5
\endpicture
$$
\caption{Parametric domain $\Omega=[a,b]\times[c,d]$, pre-images both of the network curves and of their control points ($\bullet$).}
\label{par_domain}
\end{figure}

Although we suppose the curve network is given, however we remark that a curve network satisfying the condition C2. always exists. Indeed, from the property of B-spline local support, we notice that for $r=1,\ldots,m-1$ and $s=1,\ldots,n-1$, the condition C2. is equivalent to
\begin{equation} 
P_{s}^{(r)}B_{s}(v_{s}|V)+P_{s+1}^{(r)}B_{s+1}(v_{s}|V)=Q_{r}^{(s)}B_{r}(u_{r}|U)+Q_{r+1}^{(s)}B_{r+1}(u_{r}|U).\label{inner}
\end{equation}
Since we have (\cite[Chap. X]{dB}; \cite[Sect. 1.5]{ch})
$$
\begin{array}{ll}
B_{r}(u_{r}|U)=\sigma_{r+1},   & B_{r+1}(u_{r}|U)=\sigma_{r+1}',  \\
B_{s}(v_{s}|V)=\tau_{s+1},   & B_{s+1}(v_{s}|V)=\tau_{s+1}',
\end{array}
$$
where, for $0 \leq r \leq m+1$ and $0 \leq s \leq n+1$,
\begin{equation}\label{s-t}
\begin{array}{c}
\sigma_{r+1}=\frac{h_{r+1}}{h_{r}+h_{r+1}} ,\ \ \sigma_r'=\frac{h_{r-1}}{h_{r-1}+h_r} ,\\
\tau_{s+1}=\frac{k_{s+1}}{k_{s}+k_{s+1}} ,\ \ \tau_s'=\frac{k_{s-1}}{k_{s-1}+k_s}, 
\end{array}
\end{equation}
with $h_r=u_r-u_{r-1}$, $k_s=v_s-v_{s-1}$ and $h_{-1}=h_{m+2}=k_{-1}=k_{n+2}=0$ (when in (\ref{s-t}) we have $\frac00$, we set the corresponding value equal to zero), therefore, the condition (\ref{inner}) can be written as follows
\begin{equation}
\tau_{s+1}P_{s}^{(r)}+\tau_{s+1}'P_{s+1}^{(r)}=\sigma_{r+1}Q_{r}^{(s)}+\sigma_{r+1}'Q_{r+1}^{(s)}. \label{cond111}
\end{equation}

By a similar argument, we can prove that, for $s=0,n$ and $r=1,\ldots,m-1$, the condition C2. is equivalent to
\begin{equation}
\sigma_{r+1}Q_{r}^{(0)}+\sigma_{r+1}'Q_{r+1}^{(0)}=P_0^{(r)},\ \ \ \sigma_{r+1}Q_{r}^{(n)}+\sigma_{r+1}'Q_{r+1}^{(n)}=P_{n+1}^{(r)} \label{cond222}
\end{equation}
and, for $r=0,m$ and $s=1,\ldots,n-1$, it is equivalent to
\begin{equation}
\tau_{s+1}P_{s}^{(0)}+\tau_{s+1}'P_{s+1}^{(0)}=Q_0^{(s)},\ \ \ \tau_{s+1}P_{s}^{(m)}+\tau_{s+1}'P_{s+1}^{(m)}=Q_{m+1}^{(s)}. \label{cond333}
\end{equation}
Moreover, it results
\begin{equation}
Q_0^{(0)}=P_0^{(0)},\ Q_{m+1}^{(0)}=P_0^{(m)},\ Q_{0}^{(n)}=P_{n+1}^{(0)},\ Q_{m+1}^{(n)}=P_{n+1}^{(m)}. \label{cond444}
\end{equation}

Then, since from (\ref{rete_curve}), the number of curve network control points is $(m+1)(n+2)+(m+2)(n+1)$ and, from (\ref{cond111})-(\ref{cond444}), the number of constraints is $(m+1)(n+1)$, a curve network satisfying the condition C2. always exists and  several control points can be arbitrarily chosen to define it.


\section{Construction of the $C^1$ quadratic B-spline surface interpolating the curve network}\label{sec_int}

Let $T_{mn}$ be the criss-cross triangulation of the parameter domain $\Omega=[a,b]\times[c,d]$, based on the knots $U\times V$, given in (\ref{vetU}) and (\ref{vetV}) (Fig. \ref{par_fdom}).

Let ${\cal B}_{mn}=\{B_{ij}(u,v), (i,j)\in K_{mn}\}$, with $K_{mn}=\{(i,j): 0\leq i\leq m+1,\ 0\leq j\leq n+1\}$ be the collection of $(m+2)(n+2)$ bivariate B-splines \cite{cw,DL1,S1,S3,W}, spanning the space $S_2^1(T_{mn})$, i.e. the space of all $C^1$ quadratic splines whose restriction to each triangle of $T_{mn}$ is a bivariate polynomial of total degree two. It is well known that dim $S_2^1(T_{mn})=(m+2)(n+2)-1$ \cite{cw,W}.

By means of ${\cal B}_{mn}$, we can define a $C^1$ quadratic B-spline parametric surface of the form
\begin{equation}
{\cal S}(u,v)=\left({\cal S}_x(u,v),{\cal S}_y(u,v),{\cal S}_z(u,v)\right)^{\rm T}=\sum_{i=0}^{m+1}\sum_{j=0}^{n+1}C_{ij}B_{ij}(u,v), \quad (u,v) \in \Omega, \label{superf}
\end{equation}
interpolating the curve network (\ref{rete_curve}). In order to do it, we have to get the control points $C_{ij} \in \RR^3$, $i=0,\ldots,m+1$, $j=0,\ldots,n+1$, whose pre-images in $\Omega$ are the points $(\xi_i,\eta_j)$, defined in (\ref{si}) and (\ref{tj}), respectively (Fig. \ref{par_fdom}).

\begin{figure}[ht]
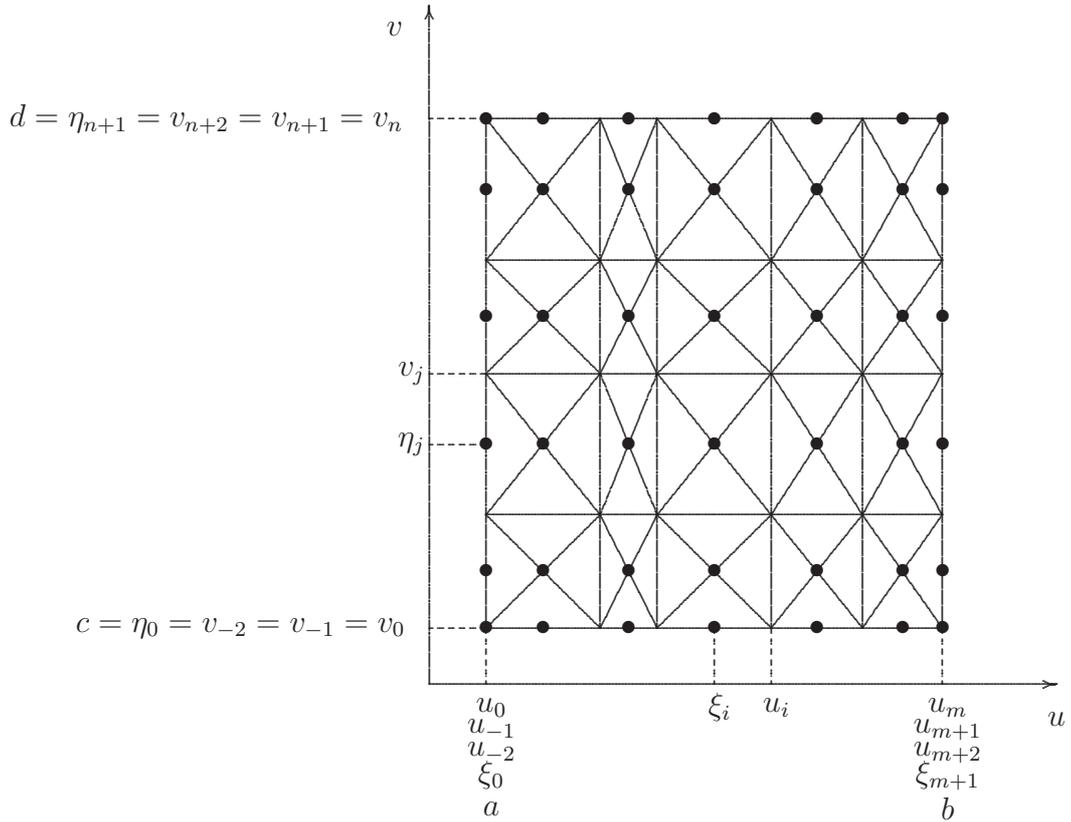

$$
\beginpicture
\setcoordinatesystem units <1.5cm,1.5cm>
\setplotarea x from -1 to 4, y from -1 to 5.5
\setlinear
\plot -0.5 -0.5  4 -0.5 /
\plot -0.5 -0.5  -0.5 4.5 /
\plot 0 0  4 0 /
\plot 0 0  0 4.5 /
\plot  4 4.5  0 4.5 /
\plot  4 4.5 4 0 /
\plot 1 0  1  4.5 /
\plot 1.5 0  1.5  4.5 /
\plot 2.5 0  2.5  4.5 /
\plot 3.3 0  3.3  4.5 /
\plot  0 1  4  1 /
\plot  0 2.25  4  2.25 /
\plot  0 3.25  4  3.25 /
\plot 0 0 1 1 /
\plot 0 1 1 2.25 /
\plot 0 2.25 1 3.25 /
\plot 0 3.25 1 4.5 /
\plot 1 0 1.5 1 /
\plot 1 1 1.5 2.25 /
\plot 1 2.25 1.5 3.25 /
\plot 1 3.25 1.5 4.5 /
\plot 1.5 0 2.5 1 /
\plot 1.5 1 2.5 2.25 /
\plot 1.5 2.25 2.5 3.25 /
\plot 1.5 3.25 2.5 4.5 /
\plot 2.5 0 3.3 1 /
\plot 2.5 1 3.3 2.25 /
\plot 2.5 2.25 3.3 3.25 /
\plot 2.5 3.25 3.3 4.5 /
\plot 3.3 0 4 1 /
\plot 3.3 1 4 2.25 /
\plot 3.3 2.25 4 3.25 /
\plot 3.3 3.25 4 4.5 /
\plot 0 1 1 0 /
\plot 0 2.25 1 1 /
\plot 0 3.25 1 2.25 /
\plot 0 4.5 1 3.25 /
\plot 1 1 1.5 0 /
\plot 1 2.25 1.5 1 /
\plot 1 3.25 1.5 2.25 /
\plot 1 4.5 1.5 3.25 /
\plot 1.5 1 2.5 0 /
\plot 1.5 2.25 2.5 1 /
\plot 1.5 3.25 2.5 2.25 /
\plot 1.5 4.5 2.5 3.25 /
\plot 2.5 1 3.3 0 /
\plot 2.5 2.25 3.3 1 /
\plot 2.5 3.25 3.3 2.25 /
\plot 2.5 4.5 3.3 3.25 /
\plot 3.3 1 4 0 /
\plot 3.3 2.25 4 1 /
\plot 3.3 3.25 4 2.25 /
\plot 3.3 4.5 4 3.25 /
\arrow <5pt> [.1,.6] from -.5 4.5 to -.5 5.5
\arrow <5pt> [.1,.6] from 4 -.5 to  5 -.5
\put {$v$} at -.8  5.3
\put {$u$} at 5  -.8
\setdashes
\setdashpattern <2pt,2pt>
\plot 0 0   0 -0.5 /
\plot 2 0   2 -0.5 /
\plot 2.5 0  2.5  -0.5 /
\plot 4 0   4 -0.5 /
\plot 0 0   -0.5 0 /
\plot 0 1.625  -0.5 1.625 /
\plot 0 2.25  -0.5 2.25 /
\plot 0 4.5  -0.5 4.5 /

\put { $u_i$} at 2.5 -0.7
\put { $v_j$} at -0.7 2.25
\put { $\xi_{i}$} at 2.0 -.7
\put { $\eta_{j}$} at -.7 1.625
\put { $a$} at 0 -1.6
\put { $\xi_0$} at 0 -1.3
\put { $u_0$} at 0 -0.7
\put { $u_{-1}$} at 0 -.9
\put { $u_{-2}$} at 0 -1.1
\put { $u_{m}$} at 4 -.7
\put { $u_{m+1}$} at 4 -.9
\put { $u_{m+2}$} at 4 -1.1
\put { $\xi_{m+1}$} at 4 -1.3
\put { $b$} at 4 -1.6
\put { $c=\eta_0=v_{-2}=v_{-1}=v_0$} at -2.2 0
\put { $d=\eta_{n+1}=v_{n+2}=v_{n+1}=v_{n}$} at -2.5 4.5
\put {$\bullet$} at 0 0
\put {$\bullet$} at 0.5 0
\put {$\bullet$} at 1.25 0
\put {$\bullet$} at  2 0
\put {$\bullet$} at  2.9 0
\put {$\bullet$} at  3.65 0
\put {$\bullet$} at  4 0
\put {$\bullet$} at 0 0.5
\put {$\bullet$} at 0.5 0.5
\put {$\bullet$} at 1.25 0.5
\put {$\bullet$} at  2 0.5
\put {$\bullet$} at  2.9 0.5
\put {$\bullet$} at  3.65 0.5
\put {$\bullet$} at  4 0.5
\put {$\bullet$} at 0 1.625
\put {$\bullet$} at 0.5 1.625
\put {$\bullet$} at 1.25 1.625
\put {$\bullet$} at  2 1.625
\put {$\bullet$} at  2.9 1.625
\put {$\bullet$} at  3.65 1.625
\put {$\bullet$} at  4 1.625
\put {$\bullet$} at 0 2.75
\put {$\bullet$} at 0.5 2.75
\put {$\bullet$} at 1.25 2.75
\put {$\bullet$} at  2 2.75
\put {$\bullet$} at  2.9 2.75
\put {$\bullet$} at  3.65 2.75
\put {$\bullet$} at  4 2.75
\put {$\bullet$} at 0 3.875
\put {$\bullet$} at 0.5 3.875
\put {$\bullet$} at 1.25 3.875
\put {$\bullet$} at  2 3.875
\put {$\bullet$} at  2.9 3.875
\put {$\bullet$} at  3.65 3.875
\put {$\bullet$} at  4 3.875
\put {$\bullet$} at 0 4.5
\put {$\bullet$} at 0.5 4.5
\put {$\bullet$} at 1.25 4.5
\put {$\bullet$} at  2 4.5
\put {$\bullet$} at  2.9 4.5
\put {$\bullet$} at  3.65 4.5
\put {$\bullet$} at  4 4.5
\endpicture
$$

\caption{The criss-cross triangulation $T_{mn}$ and pre-images of surface control points ($\bullet$).}
\label{par_fdom}
\end{figure}

Thus, we interpret the curves (\ref{rete_curve}) as isoparametric curves of ${\cal S}$, i.e.
\begin{eqnarray}
&&{\cal S}(u_r,v)=\phi_r(v),\ r=0,\ldots,m, \label{iso1} \\
&&{\cal S}(u,v_s)=\psi_s(u),\ s=0,\ldots,n \label{iso2}
\end{eqnarray}
and in Theorem \ref{teo1} we deduce the constraints that have to be satisfied by the surface control points $\{C_{ij}\}$, in order to satisfy (\ref{iso1}) and (\ref{iso2}).

\begin{theorem} \label{teo1}
The curves (\ref{rete_curve}) are isoparametric curves of ${\cal S}$ if and only if the control points $\{C_{ij}\}$ in (\ref{superf}) satisfy the following conditions:
\begin{equation}
\sigma_{r+1}C_{r,j}+\sigma_{r+1}'C_{r+1,j}=P_j^{(r)},\ r=1,\ldots,m-1, \ j=1,\ldots,n,\label{c2P}
\end{equation}
\begin{equation}
\tau_{s+1}C_{i,s}+\tau_{s+1}'C_{i,s+1}=Q_i^{(s)},\ s=1,\ldots,n-1, \ i=1,\ldots,m, \label{c2Q}
\end{equation}
\begin{equation}
C_{i0}=Q_i^{(0)},  \ i=0,\ldots,m+1, \label{edge1}
\end{equation}
\begin{equation}
C_{i,n+1}=Q_i^{(n)}, \label{edge2}
\end{equation}
\begin{equation}
C_{0j}=P_j^{(0)},  \ j=0,\ldots,n+1, \label{edge3}
\end{equation}
\begin{equation}
C_{m+1,j}=P_j^{(m)}. \label{edge4}
\end{equation}
\end{theorem}

\begin{proof} From the locality of the bivariate B-splines of ${\cal B}_{mn}$, $B_{ij}(u_r,v)\equiv 0$ for $i<r,\ i>r+1$ and any $j$. Then, from (\ref{superf}), we can write
\begin{equation}\label{cPjr}
{\cal S}(u_r,v)=\sum_{i=0}^{m+1}\sum_{j=0}^{n+1}C_{i,j}B_{i,j}(u_r,v)=\sum_{j=0}^{n+1}\left({C_{r,j}}B_{r,j}(u_r,v)+{C_{r+1,j}}B_{r+1,j}(u_r,v)\right).
\end{equation}
Considering the Bernstein-B\'ezier (BB-) coefficients of $B_{rj}(u_r,v)$, $B_{r+1,j}(u_r,v)$ and  $B_j(v|V)$, for $r=1,\ldots,m-1$ and $j=0,\ldots,n+1$ \cite{dls1,DL4}, we get 
\begin{equation}\label{equalB2B}
B_{r,j}(u_r,v)=\sigma_{r+1} B_j(v|V), \qquad     B_{r+1,j}(u_r,v)=\sigma_{r+1}' B_j(v|V). 
\end{equation}
Then, from (\ref{cPjr}) and (\ref{equalB2B}), we obtain
\begin{equation}\label{cPjr1}
{\cal S}(u_r,v)=\sum_{j=0}^{n+1}\left(\sigma_{r+1}C_{r,j}+\sigma_{r+1}'C_{r+1,j}\right) B_j(v|V). 
\end{equation}
Therefore, from (\ref{iso1}), (\ref{rete_curve}) and (\ref{cPjr1}), we get (\ref{c2P}).

Now, by the same argument, since
\begin{equation}\label{cPjr0}
{\cal S}(u_0,v)=\sum_{i=0}^{m+1}\sum_{j=0}^{n+1}C_{i,j}B_{i,j}(u_0,v)=\sum_{j=0}^{n+1}{C_{0,j}}B_{0,j}(u_0,v)
\end{equation}
and $B_{0,j}(u_0,v)=B_j(v|V)$ \cite{dls1,DL4}, from (\ref{iso1}), (\ref{rete_curve}) and (\ref{cPjr0}), we get (\ref{edge3}).

Similarly, by setting $u=u_{m}$ in (\ref{superf}), (\ref{edge4}) holds.

By using the same logical scheme, from (\ref{iso2}), we can show (\ref{c2Q}), (\ref{edge1}) and (\ref{edge2}). 
\end{proof}

\bigskip

From Theorem \ref{teo1} we can immediately get that
$$ 
\begin{array}{l} 
C_{00}=Q_0^{(0)}=P_0^{(0)}, \qquad C_{m+1,0}=Q_{m+1}^{(0)}=P_0^{(m)},\\
C_{0,n+1}=Q_0^{(n)}=P_{n+1}^{(0)}, \qquad C_{m+1,n+1}=Q_{m+1}^{(n)}=P_{n+1}^{(m)}.
\end{array}
$$

\bigskip

Now we want to show that the surface ${\cal S}$ is unique and it can be expressed as a linear combination of the $B_{ij}$'s, with coefficients depending only on the curve network control points $\{P_j^{(r)}\}_{j=0}^{n+1}$ and $\{Q_i^{(s)}\}_{i=0}^{m+1}$. 

Firstly we need the following lemma.

\begin{lemma} \label{lem1}
If curves (\ref{rete_curve}) are isoparametric curves for the surface (\ref{superf}), then for any $C_{11}\in\RR^3$
\begin{equation} 
C_{ij}=\Gamma_{ij}+(-1)^{i+j}C_{11} \frac{h_i}{h_{1}} \frac{k_j}{k_{1}} , \label{cij}
\end{equation}
$i=1,\ldots,m$, $j=1,\ldots,n$, with
\begin{equation}\label{gammaij}
\begin{array}{ll}
\Gamma_{ij}=& \displaystyle\sum_{r=1}^{i-1}(-1)^{r+1}\frac{(h_{i-r}+h_{i-r+1})h_i}{h_{i-r}h_{i-r+1}} P_j^{(i-r)}\\
& \\
           & + (-1)^i  \displaystyle\frac{h_i}{h_{1}}  \displaystyle\sum_{s=1}^{j-1}(-1)^s \frac{(k_{j-s}+k_{j-s+1})k_j}{k_{j-s}k_{j-s+1}}Q_1^{(j-s)}
\end{array}
\end{equation}
and $\sum_{\ell=1}^0\cdot=0$.
\end{lemma}

\begin{proof}
Applying repeatedly (\ref{c2Q}), for $i=1$, we can write (Fig. \ref{fig:recurs})
\begin{equation}\label{c1jrecurr}
C_{1,j}=\sum_{s=1}^{j-1}(-1)^{s+1} \frac{Q_1^{(j-s)}}{\tau_{j-s+1}'} \prod_{k=1}^{s-1} \frac{\tau_{j-k+1}}{\tau_{j-k+1}'} + (-1)^{j+1}C_{11} \prod_{s=1}^{j-1} \frac{\tau_{j-s+1}}{\tau_{j-s+1}'},
\end{equation}
where $\prod_{\ell=1}^0\cdot=1$ and $j=1,\ldots,n$.

Then, similarly, from (\ref{c2P}), it results (Fig. \ref{fig:recurs})
\begin{equation}\label{cijrecurr}
C_{i,j}=\sum_{r=1}^{i-1}(-1)^{r+1}\frac{P_j^{(i-r)}}{\sigma_{i-r+1}'} \prod_{k=1}^{r-1} \frac{\sigma_{i-k+1}}{\sigma_{i-k+1}'}+(-1)^{i+1}C_{1j} \prod_{r=1}^{i-1} \frac{\sigma_{i-r+1}}{\sigma_{i-r+1}'},
\end{equation}
where $i=1,\ldots,m$. Since 
$$
\frac{\sigma_\ell}{\sigma_\ell'}=\frac{h_\ell}{h_{\ell-1}}, \qquad  \frac{\tau_\ell}{\tau_\ell'}=\frac{k_\ell}{k_{\ell-1}},
$$
then
\begin{equation}\label{prod4}
\begin{array}{ll}
\displaystyle\prod_{r=1}^{i-1} \frac{\sigma_{i-r+1}}{\sigma_{i-r+1}'}=\frac{h_i}{h_{1}}, & \displaystyle\prod_{k=1}^{r-1} \frac{\sigma_{i-k+1}}{\sigma_{i-k+1}'}=\frac{h_i}{h_{i-r+1}},\\
& \\
\displaystyle\prod_{s=1}^{j-1} \frac{\tau_{j-s+1}}{\tau_{j-s+1}'}=\frac{k_j}{k_{1}}, & \displaystyle\prod_{k=1}^{s-1} \frac{\tau_{j-k+1}}{\tau_{j-k+1}'}=\frac{k_j}{k_{j-s+1}}.
\end{array}
\end{equation}
Taking into account (\ref{prod4}) and substituting (\ref{c1jrecurr}) into (\ref{cijrecurr}), we get (\ref{cij}) and (\ref{gammaij}).
\end{proof}

\begin{figure}[ht!]
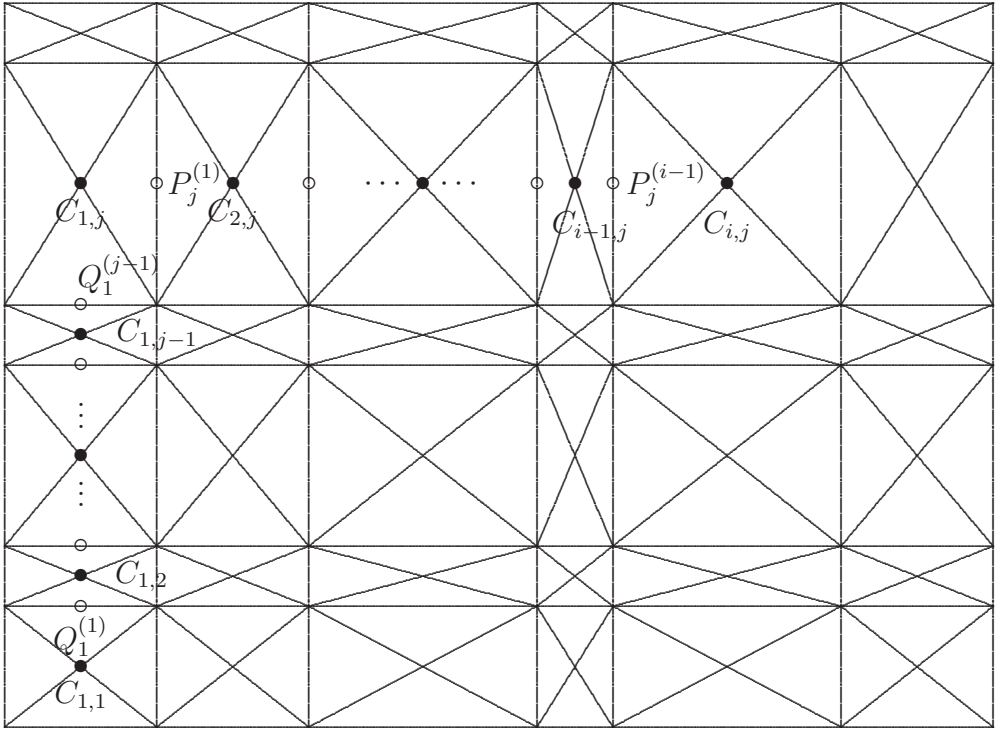

$$
\beginpicture
\setcoordinatesystem units <1cm,0.8cm>
\setplotarea x from -1 to 12, y from -2 to 10
\setlinear
\plot -1 -2  12 -2 /
\plot -1 -2  -1 10 /
\plot -1 10  12 10 /
\plot 12 10  12 -2 /
\plot 1 -2  1 10 /
\plot 3 -2  3 10 /
\plot 6 -2  6 10 /
\plot 7 -2  7 10 /
\plot 10 -2  10 10 /
\plot -1 0  12 0 /
\plot -1 1  12 1 /
\plot -1 4  12 4 /
\plot -1 5  12 5 /
\plot -1 9  12 9 /
\plot -1 -2  1 0 /
\plot -1 0  1 1 /
\plot -1 1  1 4 /
\plot -1 4  1 5 /
\plot -1 5  1 9 /
\plot -1 9  1 10 /
\plot 1 -2 3 0 /
\plot 1 0 3 1 /
\plot 1 1 3 4 /
\plot 1 4 3 5 /
\plot 1 5 3 9 /
\plot 1 9 3 10 /
\plot 3 -2 6 0 /
\plot 3 0 6 1 /
\plot 3 1 6 4 /
\plot 3 4 6 5 /
\plot 3 5 6 9 /
\plot 3 9 6 10 /
\plot 6 -2 7 0 /
\plot 6 0 7 1 /
\plot 6 1 7 4 /
\plot 6 4 7 5 /
\plot 6 5 7 9 /
\plot 6 9 7 10 /
\plot 7 -2 10 0 /
\plot 7 0 10 1 /
\plot 7 1 10 4 /
\plot 7 4 10 5 /
\plot 7 5 10 9 /
\plot 7 9 10 10 /
\plot 10 -2 12 0 /
\plot 10 0 12 1 /
\plot 10 1 12 4 /
\plot 10 4 12 5 /
\plot 10 5 12 9 /
\plot 10 9 12 10 /
\plot -1 0  1 -2 /
\plot -1 1  1 0 /
\plot -1 4  1 1 /
\plot -1 5  1 4 /
\plot -1 9  1 5 /
\plot -1 10 1 9 /
\plot 1 0  3 -2 /
\plot 1 1  3 0 /
\plot 1 4  3 1 /
\plot 1 5  3 4 /
\plot 1 9  3 5 /
\plot 1 10 3 9 /
\plot 3 0  6 -2 /
\plot 3 1  6 0 /
\plot 3 4  6 1 /
\plot 3 5  6 4 /
\plot 3 9  6 5 /
\plot 3 10 6 9 /
\plot 6 0  7 -2 /
\plot 6 1  7 0 /
\plot 6 4  7 1 /
\plot 6 5  7 4 /
\plot 6 9  7 5 /
\plot 6 10 7 9 /
\plot 7 0  10 -2 /
\plot 7 1  10 0 /
\plot 7 4  10 1 /
\plot 7 5  10 4 /
\plot 7 9  10 5 /
\plot 7 10 10 9 /
\plot 10 0  12 -2 /
\plot 10 1  12 0 /
\plot 10 4  12 1 /
\plot 10 5  12 4 /
\plot 10 9  12 5 /
\plot 10 10 12 9 /
\put {$\circ$} at 1 7
\put {$\circ$} at 3 7
\put {$\circ$} at 6 7
\put {$\circ$} at 7 7
\put {$\circ$} at 0 0
\put {$\circ$} at 0 1
\put {$\circ$} at 0 4
\put {$\circ$} at 0 5
\put {$\bullet$} at 0 7
\put {$\bullet$} at 2 7
\put {$\bullet$} at 4.5 7
\put {$\bullet$} at 6.5 7
\put {$\bullet$} at 8.5 7
\put {$\bullet$} at 0 -1
\put {$\bullet$} at 0 0.5
\put {$\bullet$} at 0 2.5
\put {$\bullet$} at 0 4.5
\put {$\cdots$} at 4 7
\put {$\cdots$} at 5 7
\put {$\vdots$} at 0 2
\put {$\vdots$} at 0 3.3
\put {{$C_{1,1}$}} at 0 -1.5
\put {{$C_{1,2}$}} at 0.8 0.5
\put {{$C_{1,j-1}$}} at 1 4.5
\put {{$C_{1,j}$}} at 0 6.5
\put {{$C_{2,j}$}} at 2 6.5
\put {{$C_{i-1,j}$}} at 6.7 6.3
\put {{$C_{i,j}$}} at 8.5 6.3
\put {{$P_{j}^{(i-1)}$}} at 7.7 7
\put {{$Q_{1}^{(j-1)}$}} at 0.5 5.5
\put {{$P_{j}^{(1)}$}} at 1.5 7
\put {{$Q_{1}^{(1)}$}} at 0 -0.5
\endpicture
$$
\caption{Path reconstruction to get $C_{ij}$ by (\ref{cij}), starting from $C_{11}$.}
\label{fig:recurs}
\end{figure}

\begin{remark}\label{rru} In case of uniform partitions, i.e. $h_i=h$ and $k_j=k$, for any $i$, $j$, the control points $\{C_{ij}\}$ in Lemma \ref{lem1} have the following simpler expression:
$$
C_{ij}=\Gamma_{ij}+(-1)^{i+j}C_{11},
$$
$i=1,\ldots,m$, $j=1,\ldots,n$, with
$$
\Gamma_{ij}= 2\left[ \sum_{r=1}^{i-1}(-1)^{r+1}P_j^{(i-r)}+ (-1)^i \sum_{s=1}^{j-1}(-1)^s Q_1^{(j-s)}\right].
$$
\end{remark}

Now we are able to give an expression of the surface ${\cal S}$, depending only on the curve network control points $\{P_j^{(r)}\}_{j=0}^{n+1}$ and $\{Q_i^{(s)}\}_{i=0}^{m+1}$.

\begin{theorem} \label{teo2}
Given the B-spline curve network (\ref{rete_curve}), satisfying the compatibility conditions C1. and C2., there is a unique surface (\ref{superf}), having (\ref{rete_curve}) as isoparametric curves. It is independent of $C_{11}$ choice and it can be written as follows:
\begin{equation}\label{Suvsurf} 
{\cal S}(u,v)= {\cal S}_b(u,v) + {\cal S}_{\Gamma}(u,v),
\end{equation}
where
$$
	\begin{array}{ll} 
{\cal S}_b(u,v)=&\displaystyle\sum_{j=0}^{n+1}\left( P_j^{(0)}B_{0j}(u,v)+P_j^{(m)}B_{m+1,j}(u,v)\right)\\
& +\displaystyle\sum_{i=1}^{m} \left(Q_i^{(0)}B_{i0}(u,v)+Q_i^{(n)}B_{i,n+1}(u,v)\right),\\
{\cal S}_{\Gamma}(u,v)=& \displaystyle\sum_{i=1}^{m}\displaystyle\sum_{j=1}^{n}\Gamma_{ij}B_{ij}(u,v),\\
\end{array}
$$
with $\Gamma_{ij}$  defined in (\ref{gammaij}).
\end{theorem}

\begin{proof}
From Theorem \ref{teo1} and Lemma \ref{lem1}, the surface (\ref{superf}) can be written as 
\begin{eqnarray} 
{\cal S}(u,v)&=&\sum_{j=0}^{n+1} \left(P_j^{(0)}B_{0j}(u,v)+P_j^{(m)}B_{m+1,j}(u,v)\right)\nonumber\\
&+&\sum_{i=1}^{m} \left(Q_i^{(0)}B_{i0}(u,v)+Q_i^{(n)}B_{i,n+1}(u,v)\right) \label{Suv1} \\
&+&\sum_{i=1}^{m}\sum_{j=1}^{n}\Gamma_{ij}B_{ij}(u,v)+\frac{C_{11}}{h_1k_1}\sum_{i=1}^{m}\sum_{j=1}^{n}(-1)^{i+j}h_ik_jB_{ij}(u,v).\nonumber
\end{eqnarray}
We recall that, in $S_2^1(T_{mn})$, the $B_{ij}$'s are linearly dependent, the dependence relationship being \cite{W,HW}:
$$
\sum_{i=1}^{m}\sum_{j=1}^{n}(-1)^{i+j}h_ik_jB_{ij}(u,v)=0.
$$
Therefore, from (\ref{Suv1}), we obtain (\ref{Suvsurf}), i.e. the interpolating surface ${\cal S}$ is unique and it is independent of $C_{11}$ choice.
\end{proof}

Now, we want to study the behaviour of the $C_{ij}$ generation process by (\ref{cij}), with respect to the round-off error growth. Linear growth of error is usually unavoidable, while exponential growth should be avoided, since this leads to unacceptable inaccuracies \cite[Chap. 1, p.32]{Bur}. In the following theorem, we show that such a growth is linear.

\begin{theorem}\label{rr2} If the sequence of partitions $\{U \times V\}$ of $\Omega$ is $A$-quasi uniform, i.e. there exists a constant $A\geq1$ such that $0 < \max_{i,j}\{h_i, k_j \} /  \min_{i,j}\{h_i, k_j \} \leq A$, then the round-off error growth is linear. \end{theorem}

\begin{proof} Suppose the $P_j^{(r)}$'s and $Q_i^{(s)}$'s are affected by some perturbations (random noise), say $\epsilon_{P_j}^{(r)}$ and $\epsilon_{Q_i}^{(s)}$, with $\left\|\epsilon_{P_j}^{(r)}\right\|_\infty$, $\left\|\epsilon_{Q_i}^{(s)}\right\|_\infty \leq \epsilon$, $\forall i,j,r,s$. 

Then, from (\ref{cij}) and (\ref{gammaij}), instead of $\{C_{ij}\}$, the following sequence $\{\overline C_{ij}\}$ is generated:
\begin{equation} 
\overline C_{ij}=\overline \Gamma_{ij}+(-1)^{i+j}C_{11} \frac{h_i}{h_{1}} \frac{k_j}{k_{1}} , \label{cijb}
\end{equation}
$i=1,\ldots,m$, $j=1,\ldots,n$, with
\begin{equation}\label{gammaijb}
\begin{array}{ll}
\overline \Gamma_{ij}=& \displaystyle\sum_{r=1}^{i-1}(-1)^{r+1}\frac{(h_{i-r}+h_{i-r+1})h_i}{h_{i-r}h_{i-r+1}} \left(P_j^{(i-r)}+\epsilon_{P_j}^{(i-r)}\right)\\
& \\
           & + (-1)^i  \displaystyle\frac{h_i}{h_{1}}  \displaystyle\sum_{s=1}^{j-1}(-1)^s \frac{(k_{j-s}+k_{j-s+1})k_j}{k_{j-s}k_{j-s+1}}\left(Q_1^{(j-s)}+\epsilon_{Q_1}^{(j-s)}\right).
\end{array}
\end{equation}
Whence, from (\ref{edge1})-(\ref{edge4}), (\ref{cijb}) and (\ref{gammaijb}), we get
$$
\begin{array}{ll}
\left\| C_{ij}- \overline C_{ij}\right\|_\infty & = \left\|\Gamma_{ij}- \overline \Gamma_{ij} \right\|_\infty  =  \left\|  \displaystyle\sum_{r=1}^{i-1}(-1)^{r+1}\frac{(h_{i-r}+h_{i-r+1})h_i}{h_{i-r}h_{i-r+1}}\epsilon_{P_j}^{(i-r)}\right. \\
  & \\
  & \left.+ (-1)^i  \displaystyle\frac{h_i}{h_{1}}  \displaystyle\sum_{s=1}^{j-1}(-1)^s \frac{(k_{j-s}+k_{j-s+1})k_j}{k_{j-s}k_{j-s+1}} \epsilon_{Q_1}^{(j-s)}\right\|_\infty.
\end{array} 
$$
Since $\{U \times V\}$ is $A$-quasi uniform, we get
$$
\begin{array}{ll}
\left\| C_{ij}- \overline C_{ij}\right\|_\infty & \leq \epsilon \left[\displaystyle\sum_{r=1}^{i-1}\frac{2(\max_{i}\{h_i\})}{(\min_{i}\{h_i\})} +  \displaystyle\frac{\max_{i}\{h_i\}}{\min_{i}\{h_i\}}  \displaystyle\sum_{s=1}^{j-1}\frac{2(\max_{j}\{k_j\})}{(\min_{j}\{k_j\})}\right]\\
  & \\
  &  \leq 2\epsilon (A(i-1)+A^2(j-1)) \leq\bar{A} (i+j-2)\epsilon,
\end{array}
$$
with $\bar{A}= 2 A^2$. Therefore, the round-off error growth is linear.
\end{proof}

We can remark that if the sequence of partitions  $\{U \times V\}$ is uniform, i.e. $A=1$, then  Theorem \ref{rr2} holds.

\section{Applications}\label{res_num}

In this section we present several applications. 

In the first one we validate the theoretical results of Theorem \ref{teo2}, i.e. we verify that the surface (\ref{superf}) is unique and it is independent of $C_{11}$ choice.

In the second one we propose a comparison with two other methods based on biquadratic tensor product B-splines. 

Finally, with the third one we want to underline that the rectangular topology does not limit the shape of the curve network. Indeed we will consider a curve net where three-sided and four-sided facet are present.

It is easy to verify numerically that the given curve networks in all above applications are isoparametric curves of the generated surfaces.

\subsection{Validation of theoretical results of Theorem \ref{teo2}}

It is well known that the control points of a spline surface usually should give an idea of its shape and its possible symmetries. In (\ref{cij})-(\ref{gammaij}) we obtain all control points of the spline surface interpolating the network (\ref{rete_curve}) depending on the curve control points and on $C_{11}$. Then different choices of $C_{11}$ lead to different surface control points, some of which do not respect such shape property. However in this application we are not interested in the surface control point shape, because here we just want to underline the uniqueness of the surface (\ref{superf}), interpolating (\ref{rete_curve}). For this reason in the following example we take two different $C_{11}$ that provide two quite different control point sets, but the same surface interpolating the curve network, as proved in Theorem \ref{teo2}.

We consider $m=18$, $n=4$, the two knot vectors $U=\{u_i\}_{i=-2}^{20}$, $V=\{v_j\}_{j=-2}^{6}$, as in (\ref{vetU}) and (\ref{vetV}), with $u_i=i$, $i=0,\ldots,18$ and $v_j=j$, $j=0,\ldots,4$ and the B-spline curve network of type (\ref{rete_curve}) shown in Fig. \ref{figg4} with given control points $\{P_j^{(r)}\}_{j=0}^5$, $r=0,\dots,18$ and $\{Q_i^{(s)}\}_{i=0}^{19}$, $s=0,\dots,4$, $P_j^{(r)}$, $Q_i^{(s)} \in \RR^3$.

The same quadratic spline surface (\ref{superf}), interpolating the above curve net and  defined by the bivariate B-splines spanning $S_{2}^{1}(T_{18,4})$, is obtained both assuming $C_{11}=(1.5,1.5,32)$ (Fig. \ref{figg5}$(a)$) and assuming $C_{11}=(1,0,30)$  (Fig. \ref{figg5}$(b)$), as we have also numerically verified.
Moreover, we can note how very different are the surface control points corresponding to the different choices of $C_{11}$. In particular, in order to obtain a control net that gives an idea of the surface shape, we choose $C_{11}$ by minimizing the area of the bilinear surface representing the control net, as done for obtaining the net shown in Fig. \ref{figg5}$(a)$.

\begin{figure}[!ht] 
\centering\includegraphics[height=7cm]{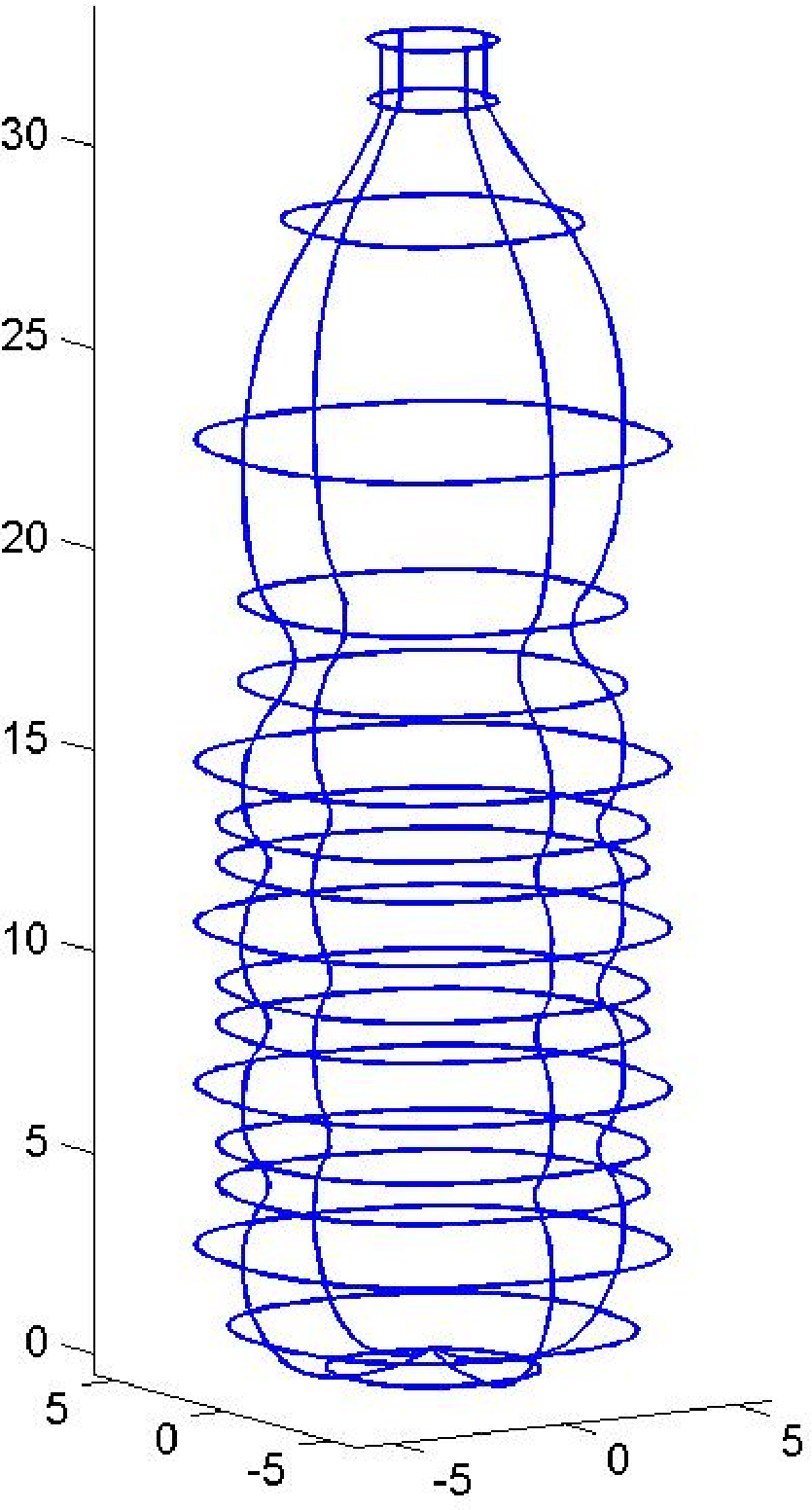}
\caption{The curve network.}
\label{figg4}
\end{figure}

\begin{figure}[ht]
\begin{minipage}{60mm} 
\centering\includegraphics[height=7cm]{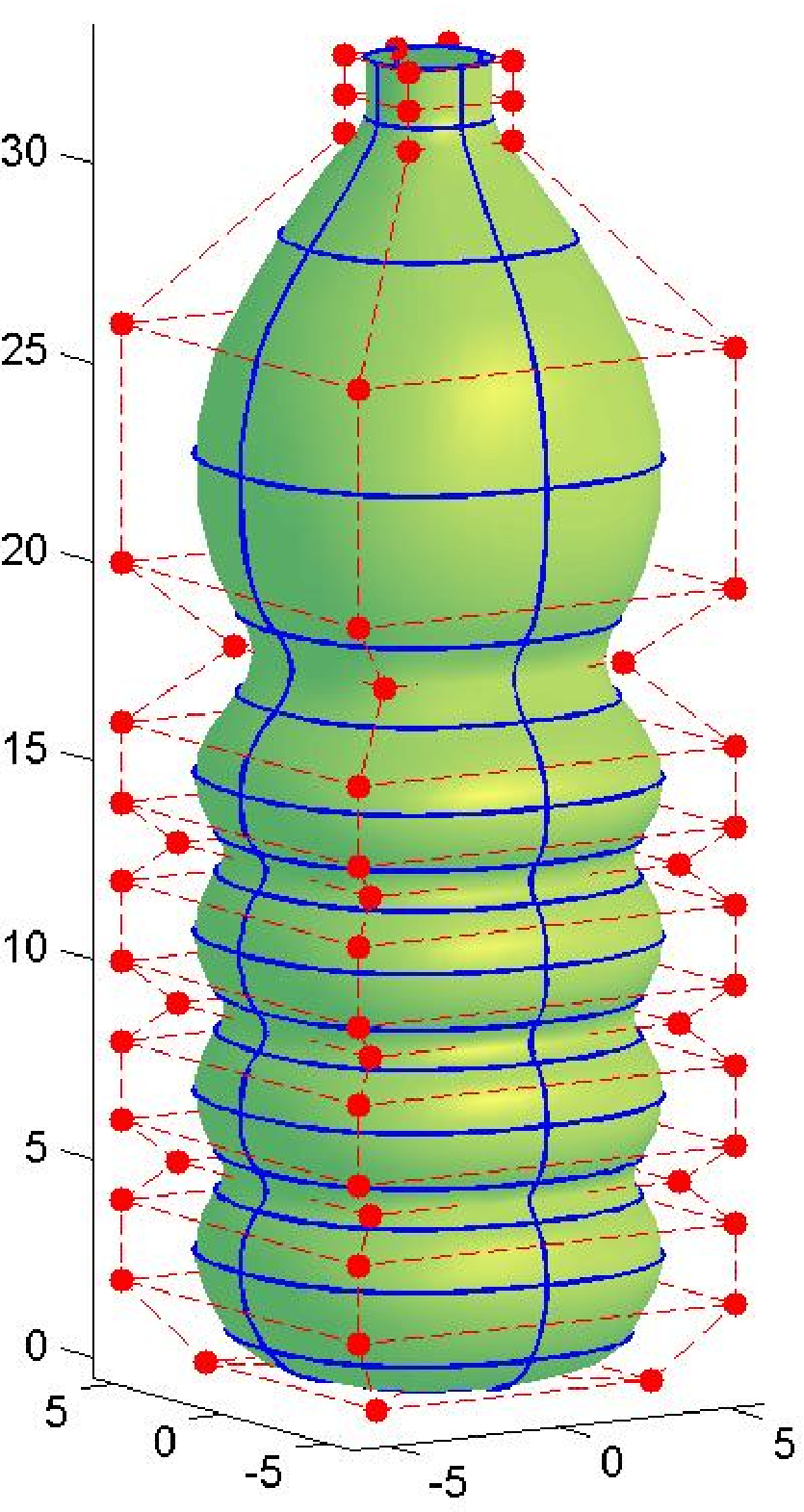}
\centerline{$(a)$}
\end{minipage}
\hfil
\begin{minipage}{60mm}
\centering\includegraphics[height=7cm]{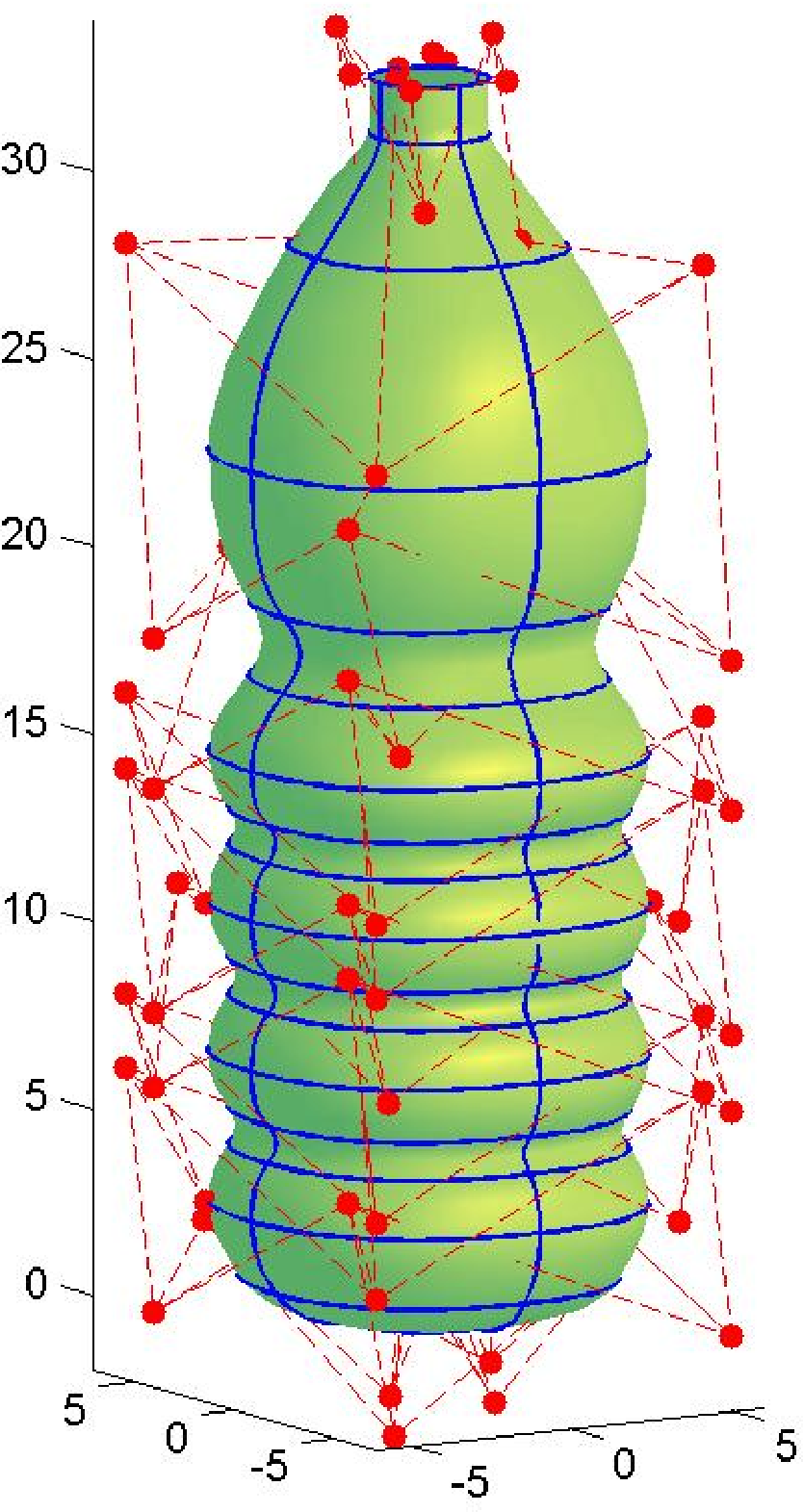}
\centerline{$(b)$}
\end{minipage}
\caption{The surfaces ${\cal S}$ interpolating the curve network with $(a)$ $C_{11}=(1.5,1.5,32)$ and $(b)$ $C_{11}=(1,0,30)$.}
\label{figg5}
\end{figure}

\subsection{Comparisons with other spline methods}

In this section we compare our method with two other ones, based on classical biquadratic tensor product B-splines: i) the first one provides a tensor product B-spline surface after a minimization process; ii) the second one consists in the construction of a Gordon-type $C^1$ biquadratic B-spline surface, \cite[Chap. 10]{pie}.

\bigskip

i) We assume the classical biquadratic tensor product B-splines on $\{U\times V\}$ as blending functions for the surface, interpolating the curve network (\ref{rete_curve}).

Let $R_{mn}$ be the rectangular partition of $\Omega=[a,b]\times[c,d]$, based on the knots $\{U\times V\}$, defined in (\ref{vetU}), (\ref{vetV}), and let $S_{2,2}^{1,1}(R_{mn})$ be the space of all biquadratic $C^1$ tensor product splines, whose restriction to each subrectangle of $R_{mn}$ is a bivariate polynomial of coordinate degree two. It is well known that dim $S_{2,2}^{1,1}(R_{mn})=(m+2)(n+2)$ and its B-spline basis is the set $\{B^*_{ij}(u,v)=B_i(u|U) \cdot B_j(v|V),\ i=0, \ldots, m+1, \ j=0, \ldots, n+1\}$, with $B_i(u|U)$ and $B_j(v|V)$ defined in Section \ref{Sect2}.

In this case, we want to construct a $C^1$ biquadratic spline surface 
\begin{equation}
{\cal T}(u,v)=\left({\cal T}_x(u,v),{\cal T}_y(u,v),{\cal T}_z(u,v)\right)^{\rm T}=\sum_{i=0}^{m+1}\sum_{j=0}^{n+1}C^*_{ij}B^*_{ij}(u,v), \label{superf_tp}
\end{equation}
interpolating the curve network (\ref{rete_curve}).

We can easily verify that results analogous to those of Theorem \ref{teo1} and Lemma \ref{lem1} still hold for ${\cal T}$, by substituting $C_{ij}$ with $C^*_{ij}$ and $B_{ij}(u,v)$ with $B^*_{ij}(u,v)$, while the corresponding results of Theorem \ref{teo2} do not hold any more. Indeed, since for tensor product quadratic B-splines 
\begin{equation}\label{ss1}
{\cal T}_1(u,v)=\sum_{i=1}^{m}\sum_{j=1}^{n}(-1)^{i+j}h_ik_jB^*_{ij}(u,v)
\end{equation}
is different from zero, because the $B^*_{ij}$'s are linearly independent, then we can write
\begin{equation}\label{Suv1tp}
{\cal T}(u,v)= {\cal T}_b(u,v) + {\cal T}_{\Gamma}(u,v) + \frac{C^*_{11}}{h_1k_1} {\cal T}_{1}(u,v),
\end{equation}
with ${\cal T}_1$ defined in (\ref{ss1}) and
$$
\begin{array}{ll} 
{\cal T}_b(u,v)=&\displaystyle\sum_{j=0}^{n+1}\left(P_j^{(0)}B^*_{0j}(u,v)+P_j^{(m)}B^*_{m+1,j}(u,v)\right)\\
& +\displaystyle\sum_{i=1}^{m} \left(Q_i^{(0)}B^*_{i0}(u,v)+Q_i^{(n)}B^*_{i,n+1}(u,v)\right),\\
{\cal T}_{\Gamma}(u,v)=& \displaystyle\sum_{i=1}^{m}\displaystyle\sum_{j=1}^{n}\Gamma_{ij}B^*_{ij}(u,v),\\
\end{array}
$$
where $\Gamma_{ij}$ is defined as in (\ref{gammaij}).

Therefore, the surface ${\cal T}$ is not unique, because it depends on the choice of $C^*_{11} \in \RR^3$, that can be performed imposing several kinds of constraints.

Here, we choose $C^*_{11}$ such that the surface is  ``smooth'', i.e. the energy of a thin plate described by the surface is minimal, as reported in literature, for instance in \cite{had}. Therefore, we can consider an approximation $\Pi$ of the thin plate energy given by
$$
\Pi=\int_a^b\int_c^d [({\cal T}^{uu}(u,v))^2+2({\cal T}^{uv}(u,v))^2+({\cal T}^{vv}(u,v))^2]{\rm d}u{\rm d}v, 
$$
with
$$
{\cal T}^{uu}(u,v)=\frac{\partial^2{\cal T}(u,v)}{\partial u^2}, \, {\cal T}^{uv}(u,v)=\frac{\partial^2{\cal T}(u,v)}{\partial u\partial v}, \, {\cal T}^{vv}(u,v)=\frac{\partial^2{\cal T}(u,v)}{\partial v^2}
$$ 
and choose $C^*_{11}$ by minimizing $\Pi$ with respect to $C^*_{11}$.  By setting $\frac{{\rm d} \Pi}{{\rm d} C^*_{11}}=0$, we obtain
\begin{equation}\label{cc11}
C^*_{11}=-h_1k_1\frac{\int_a^b\int_c^d [({\cal T}_b^{uu}+{\cal T}_\Gamma^{uu}){\cal T}_1^{uu}+2({\cal T}_b^{uv}+{\cal T}_\Gamma^{uv}){\cal T}_1^{uv}+({\cal T}_b^{vv}+{\cal T}_\Gamma^{vv}){\cal T}_1^{vv}]}{\int_a^b\int_c^d [({\cal T}_1^{uu})^2+2({\cal T}_1^{uv})^2+({\cal T}_1^{vv})^2]}.
\end{equation}
It is easy to verify that such a point is the one minimizing $\Pi$. 

Since in each subrectangle $[u_i,u_{i+1}]\times [v_j,v_{j+1}]$, ${\cal T}$ is a biquadratic polynomial surface, we can exactly evaluate the integrals appearing in (\ref{cc11}), for instance by using a composite tensor-product Gauss-Legendre quadrature formula with $3\times3$ nodes.

We notice that results similar to the ones of Remark \ref{rru} and Theorem \ref{rr2} also hold for the surface ${\cal T}$, given in (\ref{Suv1tp}). 

\bigskip

ii) The second considered approach consists in the construction of a Gordon-type $C^1$ biquadratic B-spline surface, written as a boolean sum \cite[Chap. 10]{pie}
\begin{equation}\label{Sgor}
{\cal G}(u,v)={\cal L}_1(u,v)+{\cal L}_2(u,v)-{\cal T}_{\cal G}(u,v),
\end{equation}
where ${\cal L}_1$, ${\cal L}_2$ are two skinned surfaces and ${\cal T}_{\cal G}$ is a tensor product surface interpolating a given set of points. More precisely, in our case, given the curve network (\ref{rete_curve}), satisfying the compatibility conditions C1. and C2., such surfaces are defined in the following way:
\begin{itemize}
\item the skinned surface \cite[Chap. 10]{pie}
$$
{\cal L}_1(u,v)=\sum_{i=0}^m\sum_{j=0}^{n+1}C^{{\cal L}_1}_{ij}\bar{B}_i(u|\bar{U})\cdot B_j(v|V),
$$
defined on $\bar{U}\times V$, with  $\bar{U}$ the knot vector obtained from $U$, given in (\ref{vetU}), by the averaging technique \cite[Chap. 9]{pie}, interpolates the curves $\phi_r(v)$, given in (\ref{rete_curve}), at the nodes $u_r$, $r=0,\ldots,m$, i.e. ${\cal L}_1(u_r,v)=\phi_r(v)$;

\item similarly, the skinned surface
$$
{\cal L}_2(u,v)=\sum_{i=0}^{m+1}\sum_{j=0}^n C^{{\cal L}_2}_{ij}B_i(u|U)\cdot\bar{B}_j(v|\bar{V}),
$$
defined on $U\times\bar{V}$, with  $\bar{V}$ the knot vector obtained from $V$, given in (\ref{vetV}), by the averaging technique, interpolates the curves $\psi_s(u)$, given in (\ref{rete_curve}), at the nodes $v_s$, $s=0,\ldots,n$, i.e. ${\cal L}_2(u,v_s)=\psi_s(u)$;

\item the tensor product surface
$$
{\cal T}_{\cal G}(u,v)=\sum_{i=0}^{m}\sum_{j=0}^n C^{{\cal T}_{\cal G}}_{ij}\bar{B}_i(u|\bar{U})\cdot\bar{B}_j(v|\bar{V}),
$$
defined on $\bar{U}\times\bar{V}$, interpolates the points $I_{rs}$, given in (\ref{incroci}), at the nodes $(u_r,v_s)$, $r=0,\ldots,m$, $s=0,\ldots,n$, i.e. ${\cal T}_{\cal G}(u_r,v_s)=I_{rs}$.
\end{itemize}

Following \cite[Chap. 10]{pie}, we give a standard B-spline representation of the surface ${\cal G}(u,v)$, given in (\ref{Sgor}). In order to do it, since the three surfaces ${\cal L}_1$, ${\cal L}_2$ and ${\cal T}_{\cal G}$ are defined by B-spline functions belonging to different spline spaces, we need the three surfaces to be compatible in the B-spline sense, i.e. defined on the same knot vectors. Then, by merging $U,\bar{U}$ and $V,\bar{V}$, we obtain the two knot vectors $\tilde{U}$ and $\tilde{V}$ of length $2m+3$ and $2n+3$, respectively, and we apply the knot refinement algorithm to ${\cal L}_1$, ${\cal L}_2$ and ${\cal T}_{\cal G}$. 

Therefore, we get the new control points $\{\tilde{C}^{{\cal L}_1}_{ij}\}$, $\{\tilde{C}^{{\cal L}_2}_{ij}\}$ and $\{\tilde{C}^{{\cal T}_{\cal G}}_{ij}\}$ of the three surfaces and we can write
$$
{\cal G}(u,v)=\sum_{i=0}^{2m-1}\sum_{j=0}^{2n-1}\tilde{C}_{ij}\tilde{B}_{ij}(u,v),
$$
where $\tilde{B}_{ij}(u,v)=\tilde{B}_i(u|\tilde{U})\cdot\tilde{B}_j(v|\tilde{V})$, with $\tilde{B}_i(u|\tilde{U})$ and $\tilde{B}_j(v|\tilde{V})$ univariate quadratic B-splines defined on the knot partitions $\tilde{U}$ and $\tilde{V}$, respectively, and 
\begin{equation} \label{Ctilde}
\tilde{C}_{ij}=\tilde{C}^{{\cal L}_1}_{ij}+\tilde{C}^{{\cal L}_2}_{ij}-\tilde{C}^{{\cal T}_{\cal G}}_{ij}.
\end{equation}

\bigskip

Now, we propose an example, in order to compare our method with the above ones, from the graphical and numerical points of view.

Given a bidirectional curve network, we generate the three surfaces ${\cal S}$, defined in (\ref{superf}), ${\cal T}$, defined in (\ref{superf_tp}), and ${\cal G}$, defined in (\ref{Sgor}), with $m=n=3$, $\Omega=[0,1]\times[0,1]$, $U=V=\{0,0,0,\frac{1}{3},\frac{2}{3},1,1,1\}$. In this case
$$
\begin{array}{l} 
\bar{U}=\bar{V}=\{0,0,0,\frac{1}{2},1,1,1\}, \\
\tilde{U}=\tilde{V}=\{0,0,0,\frac{1}{3},\frac{1}{2},\frac{2}{3},1,1,1\}. \\
\end{array}
$$

In Fig. \ref{figgg6} we show the curve network and the B-spline surface ${\cal S}$, where $C_{11}=(1,0,3)$ has been arbitrarily chosen, while in Fig. \ref{figgg7} we display the biquadratic B-spline surface ${\cal T}$ with $C^*_{11}=(1.11,-0.15,3)$ obtained by (\ref{cc11}) and the Gordon-type biquadratic B-spline surface ${\cal G}$, where the control points are written as in (\ref{Ctilde}).

\begin{figure}[ht]
\begin{minipage}{60mm} 
\centering\includegraphics[width=6cm]{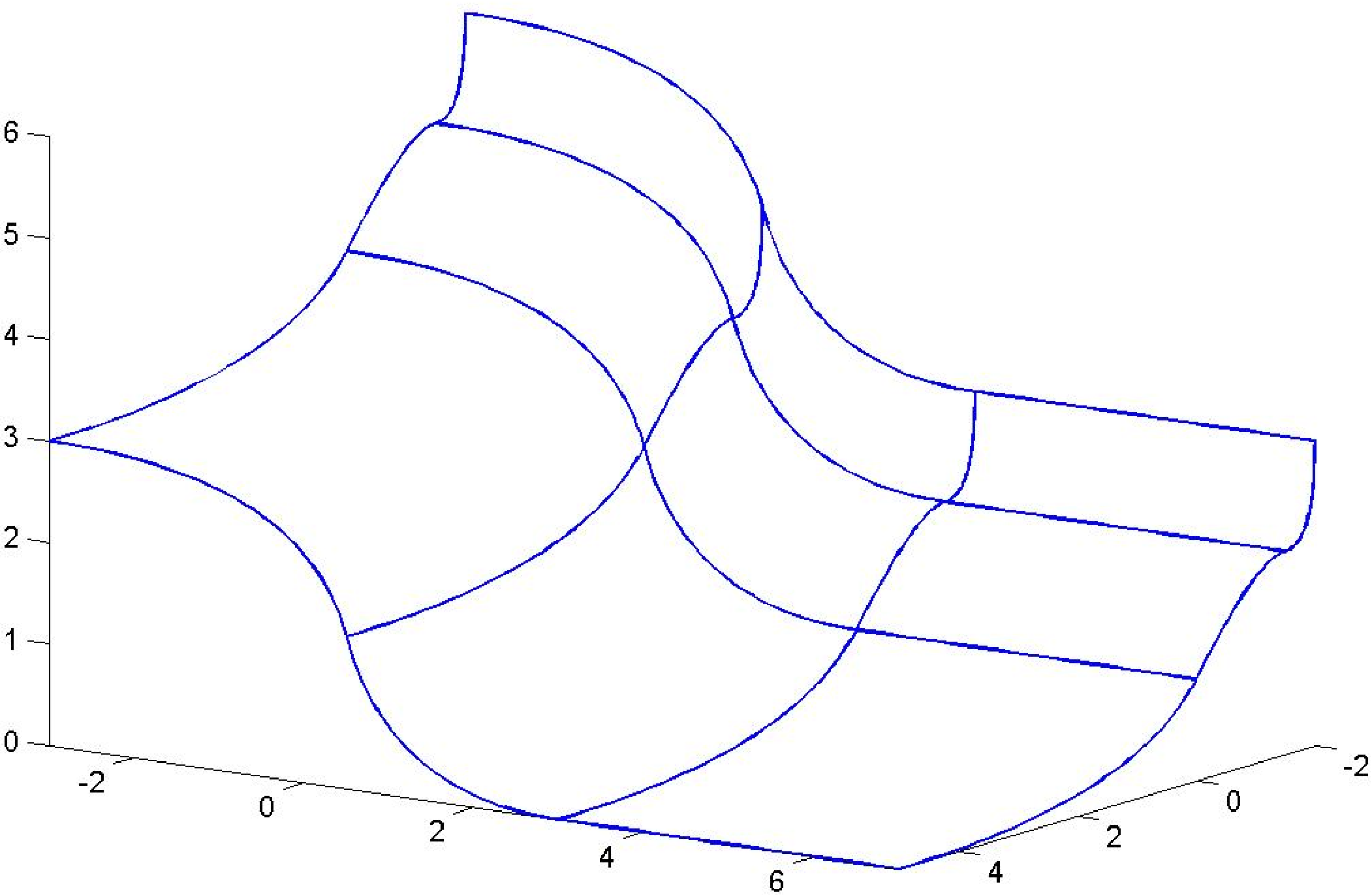}
\centerline{$(a)$}
\end{minipage}
\hfil
\begin{minipage}{60mm} 
\centering\includegraphics[width=6cm]{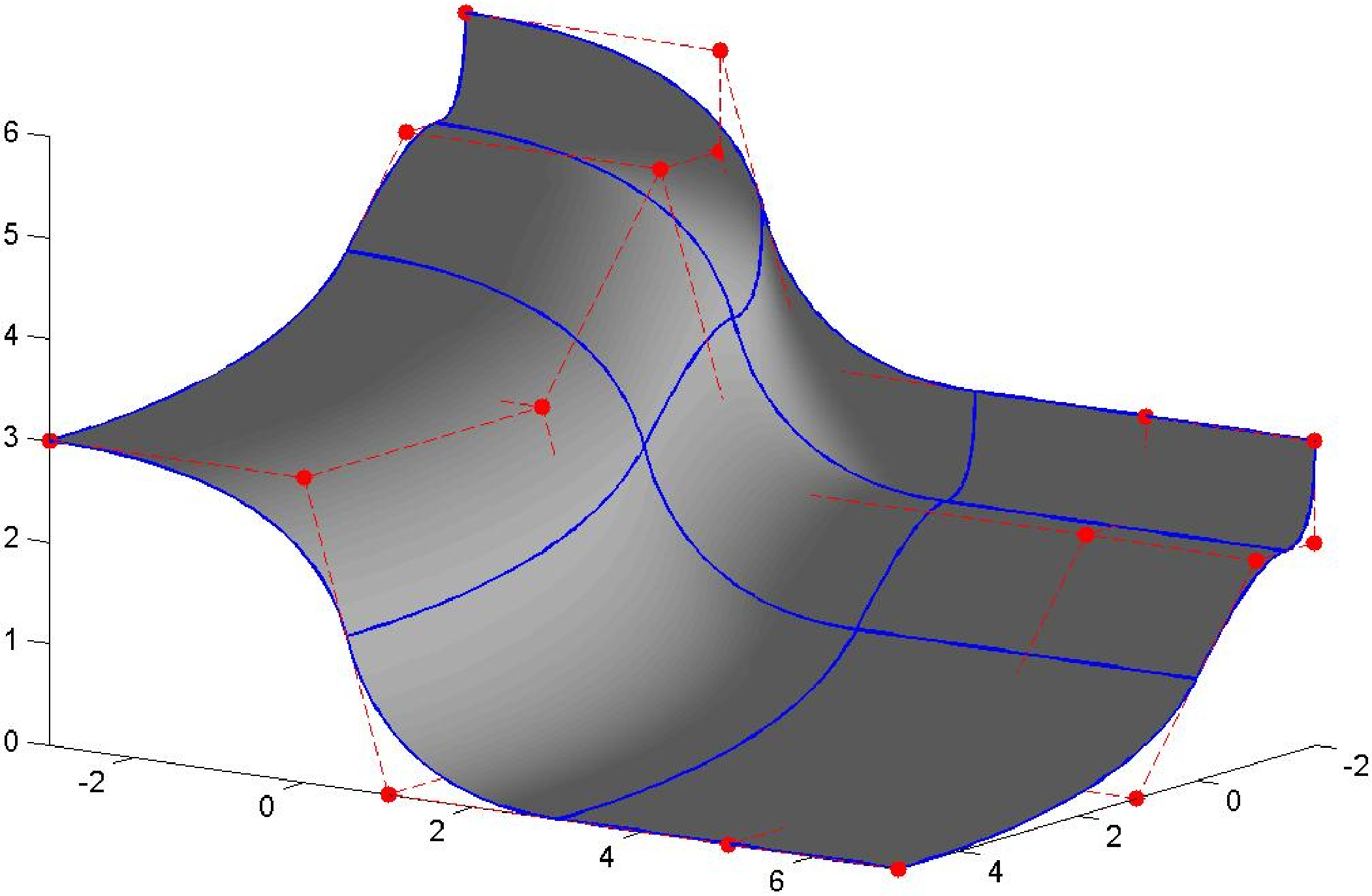}
\centerline{$(b)$}
\end{minipage}
\caption{$(a)$ The curve network and $(b)$ the B-spline surface ${\cal S}$, defined by B-spline functions of $S_{2}^{1}(T_{3,3})$ and interpolating the curve network $(a)$.}
\label{figgg6}
\end{figure}

\begin{figure}[ht]
\begin{minipage}{60mm} 
\centering\includegraphics[width=6cm]{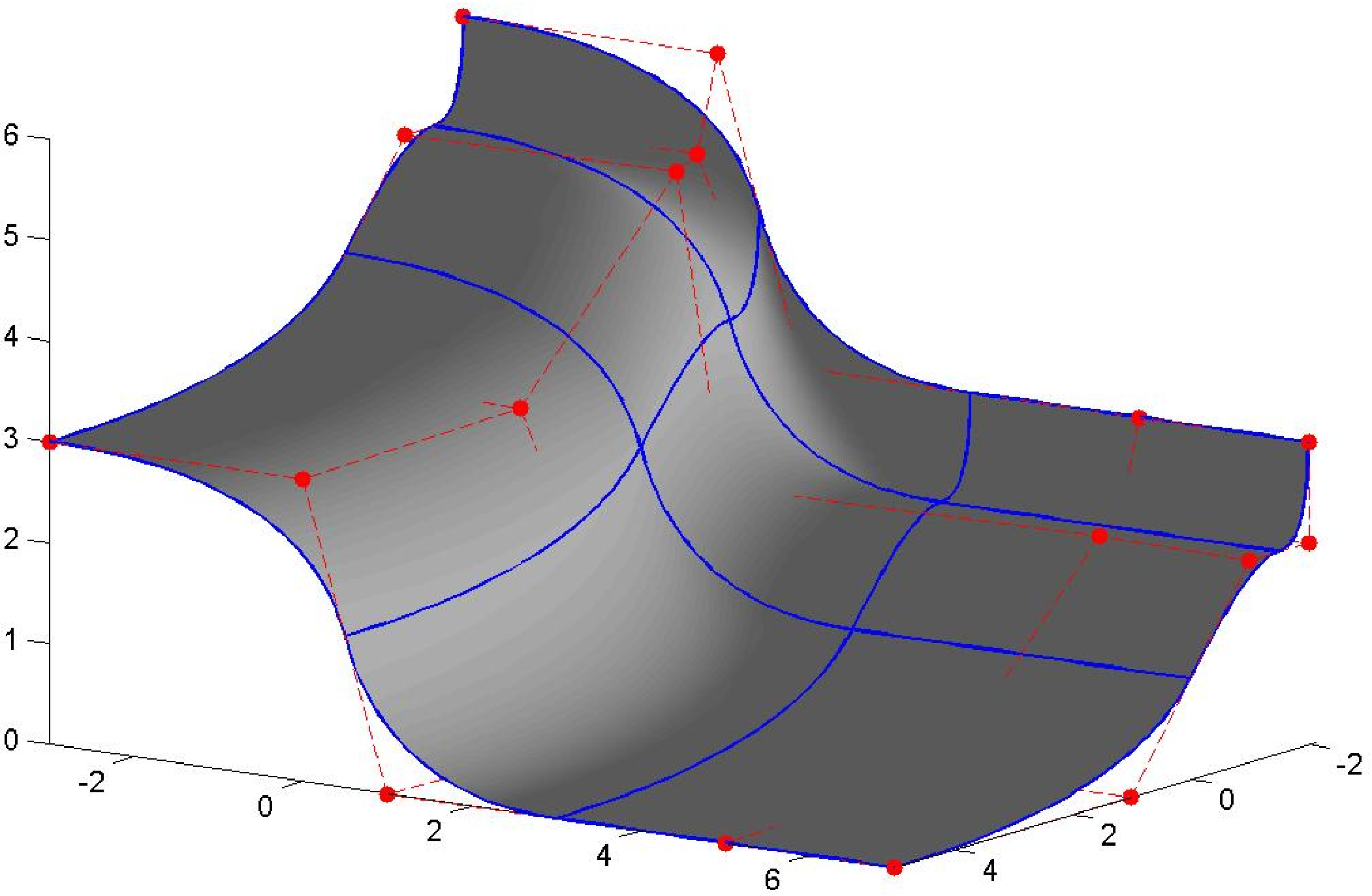}
\centerline{$(c)$}
\end{minipage}
\hfil
\begin{minipage}{60mm} 
\centering\includegraphics[width=6cm]{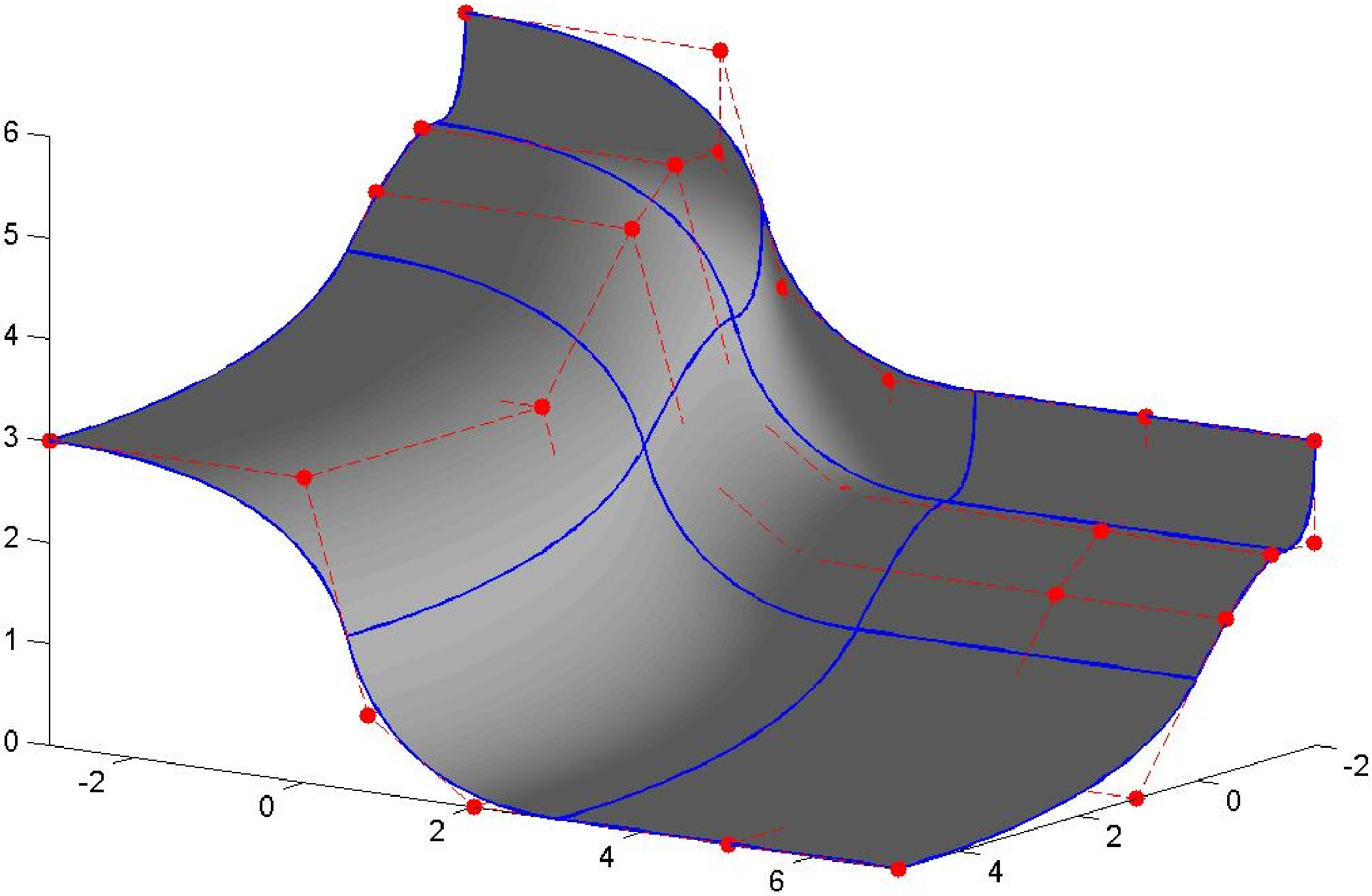}
\centerline{$(d)$}
\end{minipage}
\caption{$(c)$ The surface ${\cal T}$, defined by B-spline functions of $S_{2,2}^{1,1}(R_{3,3})$, and $(d)$ the surface ${\cal G}$, defined by B-spline functions of $S_{2,2}^{1,1}(R_{8,8})$, interpolating the curve network $(a)$ of Fig. \ref{figgg6}.}
\label{figgg7}
\end{figure}
 
We can remark that the shapes of the generated surfaces are comparable from  the graphical point of view. 

Moreover, we have also compared them from a numerical point of view.
Indeed, let G be a $57 \times 57$ uniform rectangular grid of evaluation points in $\Omega$. Then we get
$$
||{\cal S}-{\cal T}||_\infty=\max_{\nu=x,y,z}\{\max_{(t_1,t_2)\in G}|{\cal S}_\nu(t_1,t_2)-{\cal T}_\nu(t_1,t_2)|\}=3.7879(-2),
$$
$$
||{\cal S}-{\cal G}||_\infty=\max_{\nu=x,y,z}\{\max_{(t_1,t_2)\in G}|{\cal S}_\nu(t_1,t_2)-{\cal G}_\nu(t_1,t_2)|\}=4.4409(-15),
$$
$$
||{\cal T}-{\cal G}||_\infty=\max_{\nu=x,y,z}\{\max_{(t_1,t_2)\in G}|{\cal T}_\nu(t_1,t_2)-{\cal G}_\nu(t_1,t_2)|\}=3.7879(-2).
$$

We remark that the surfaces ${\cal S}$ and ${\cal G}$ seem to have the same behaviour.
However, we have to note that the constructions both of ${\cal T}$ and of ${\cal G}$ are computationally more expensive than the one of the surface ${\cal S}$, here proposed. Indeed,  the generation of ${\cal T}$ requires a minimization process to evaluate  $C^*_{11}$ in (\ref{cc11}), while for the generation of ${\cal G}$ we have to construct three different surfaces by using B-splines  belonging to different spline spaces, so that the bivariate resulting surface is not defined on the rectangular grid associated to the knot vectors of the curve network, since knot refinement is required.

\subsection{Curve network with 3 and 4-sided facet}

Here we want to underline that the rectangular topology of the parametric domain $\Omega$ does not constrain the curve network to have a 4-sided facet.

In the following test we interpolate the B-spline curve network shown in Fig. \ref{figg6}$(a)$, where three-sided and four-sided facet are present.

Considering $m=5$, $n=3$, $\Omega=[0,5]\times[0,3]$ and the two knot vectors $U=\{u_i\}_{i=-2}^{7}$, $V=\{v_j\}_{j=-2}^{5}$, as in (\ref{vetU}) and (\ref{vetV}), with $u_i=i$, $i=0,\ldots,5$ and $v_j=j$, $j=0,\ldots,3$, by means of the B-spline functions in $S_{2}^{1}(T_{5,3})$ we construct the interpolating B-spline surface (\ref{superf}) that is shown in Fig. \ref{figg6}$(b)$, with $C_{11}=(1.5,1,1)$ arbitrarily chosen. 

Indeed, if in (\ref{rete_curve}) we assume $Q_i^{(0)}=Q^{(0)}$ and $Q_i^{(3)}=Q^{(3)}$, $i=0,\ldots,6$, then we get
$$
\begin{array}{l} 
\psi_0(u)=\displaystyle\sum_{i=0}^{6}Q_i^{(0)}B_i(u|U)=Q^{(0)}\sum_{i=0}^{6}B_i(u|U)=Q^{(0)},\\
\psi_3(u)=\displaystyle\sum_{i=0}^{6}Q_i^{(3)}B_i(u|U)=Q^{(3)}\sum_{i=0}^{6}B_i(u|U)=Q^{(3)}
\end{array}
$$
with $u\in [0,5]$. This means that all curves $\phi_r(v),\ r=0,\ldots,5,\ v\in [0,3]$ begin at $Q^{(0)}$ and they end at $Q^{(3)}$, i.e. from (\ref{incroci})
$$
\phi_r(v_s)=\psi_s(u_r)=Q^{(s)}
$$
with $s=0,3$ and $r=0,\ldots,5$.

\begin{figure}[ht]
\begin{minipage}{60mm} 
\centering\includegraphics[width=6cm]{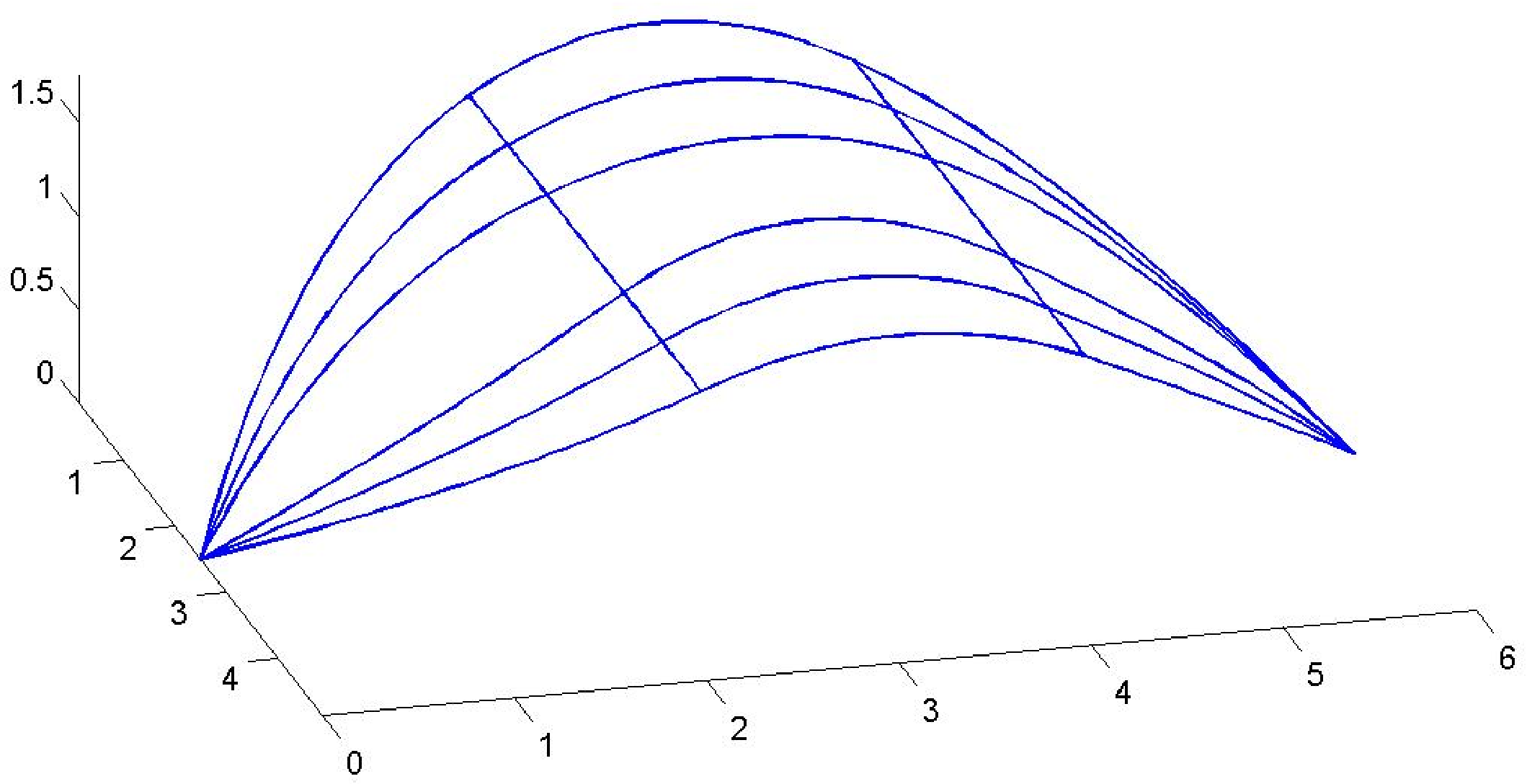}
\centerline{$(a)$}
\end{minipage}
\hfil
\begin{minipage}{60mm} 
\centering\includegraphics[width=6cm]{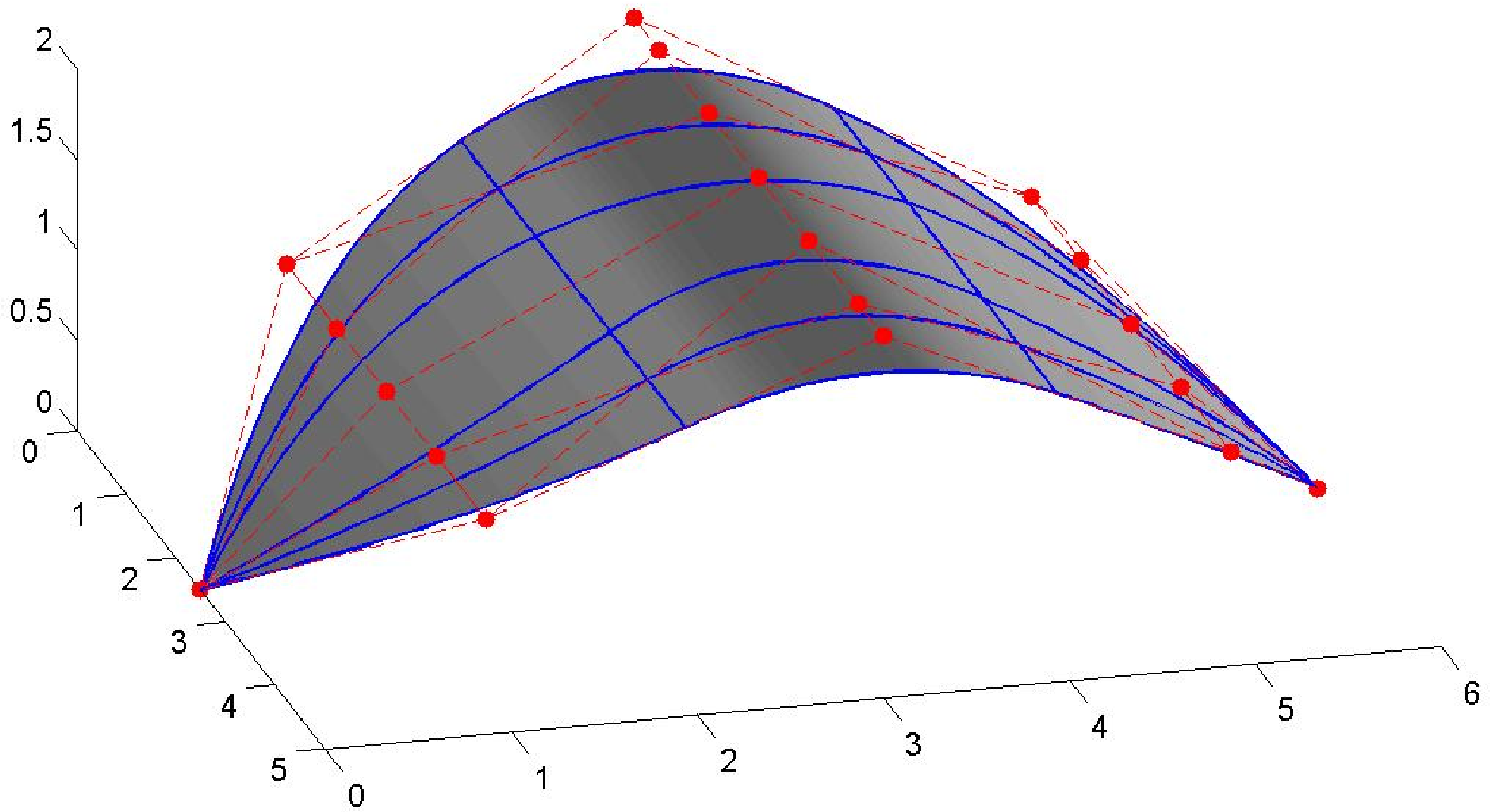}
\centerline{$(b)$}
\end{minipage}
\caption{$(a)$ The curve network and $(b)$ the surface ${\cal S}$ interpolating it.}
\label{figg6}
\end{figure}

In Fig. \ref{figg7} we show another example with  $m=5$, $n=4$, $\Omega=[0,5]\times[0,4]$ and the two knot vectors $U=\{u_i\}_{i=-2}^{7}$, $V=\{v_j\}_{j=-2}^{6}$, as in (\ref{vetU}) and (\ref{vetV}) with $u_i=i$, $i=0,\ldots,5$, $v_j=j$, $j=0,\ldots,4$, and $C_{11}=(0.5,0.5,3.5)$.

\begin{figure}[ht]
\begin{minipage}{60mm} 
\centering\includegraphics[width=4.5cm]{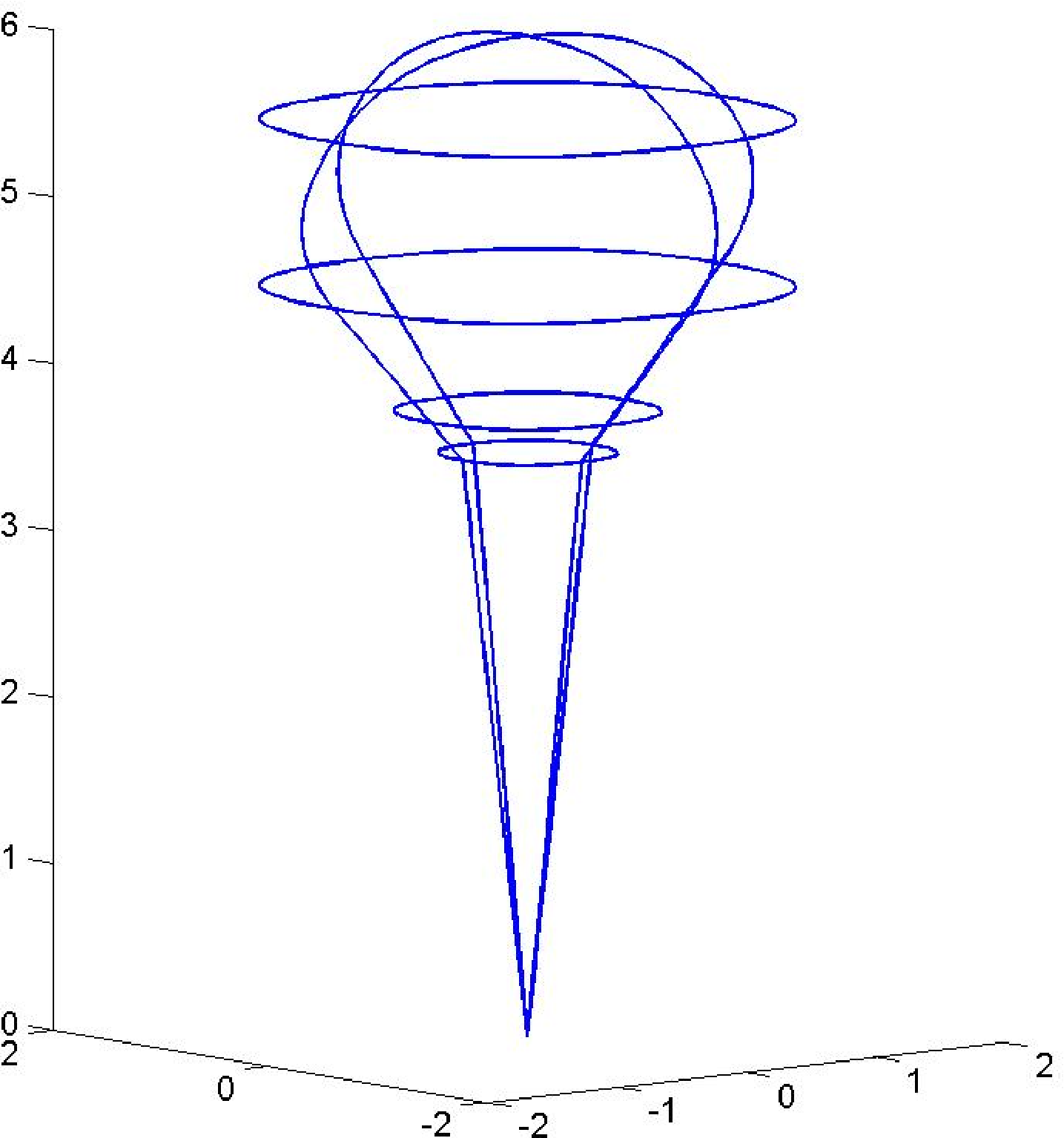}
\centerline{$(a)$}
\end{minipage}
\hfil
\begin{minipage}{60mm} 
\centering\includegraphics[width=8cm]{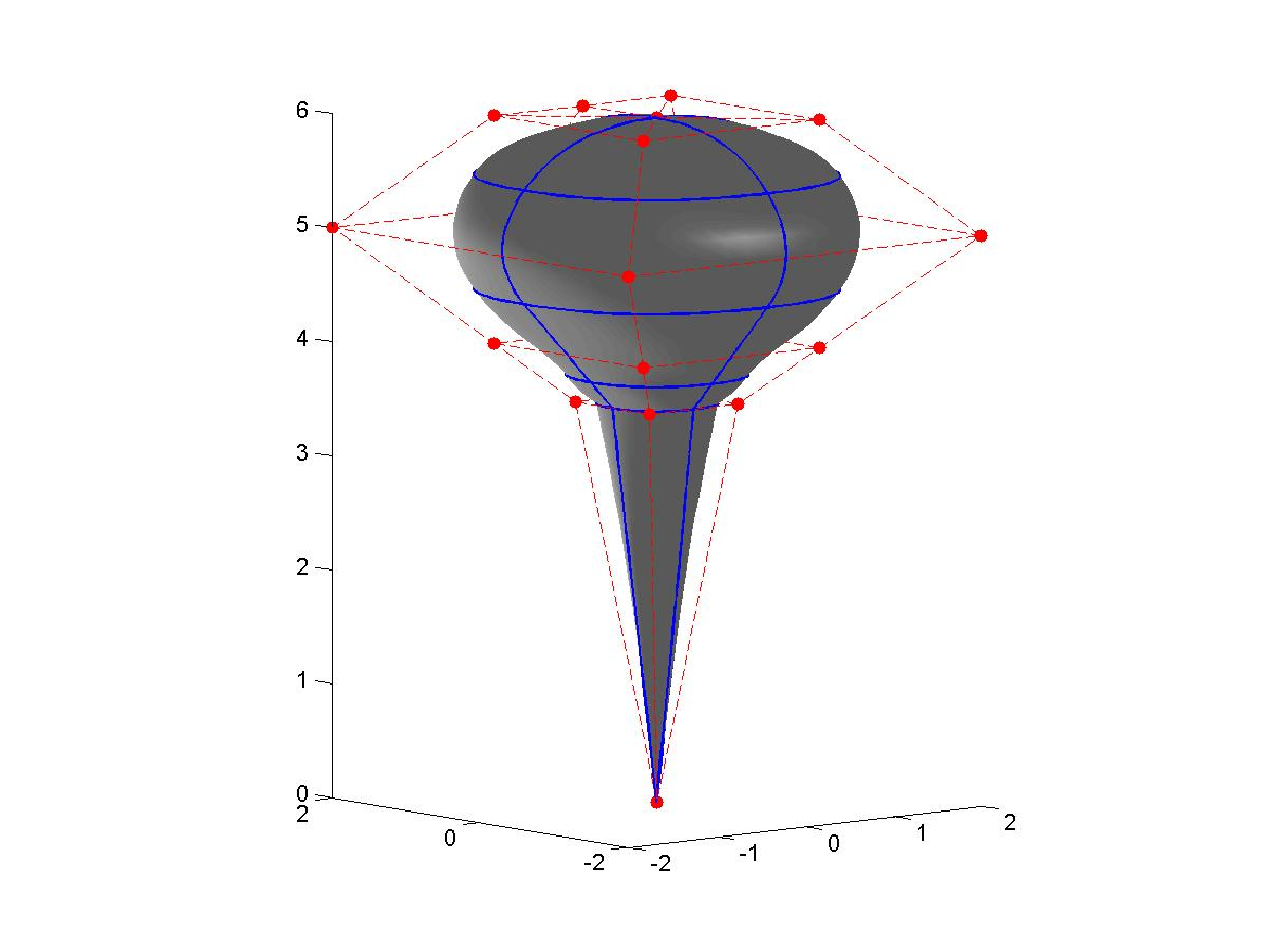}
\centerline{$(b)$}
\end{minipage}
\caption{$(a)$ The curve network and $(b)$ the surface ${\cal S}$ interpolating it.}
\label{figg7}
\end{figure}

\section{Final remarks}

In this paper we have investigated the problem of interpolating a B-spline curve network, in order to create a surface satisfying such a constraint and defined by blending functions spanning the space $S_2^1(T_{mn})$. We have proved the surface existence and uniqueness, providing a constructive algorithm for its generation.

We have also presented numerical and graphical results and comparisons with tensor product B-spline surfaces. 

According to \cite[Chap. 10]{pie}, we have restricted the curves of the network to be nonrational. Theoretically, the construction, discussed in this paper, can be carried out in homogeneous space for rational curves, but it is generally not practical to do so and the three-dimensional results can be unpredictable. If rational curves are involved, it is recommended using constrained approximation techniques to obtain nonrational approximations of the rational curves to within necessary tolerances.

Finally, we remark that an interesting open problem could be the construction of B-spline surfaces interpolating a curve network of the following form
$$
\begin{array}{l} 
\phi_r(v)=\displaystyle\sum_{j=0}^{n+1}P_j^{(r)}B_j(v|V)\quad\quad r=0,\dots,R,\\
\psi_s(u)=\displaystyle\sum_{i=0}^{m+1}Q_i^{(s)}B_i(u|U)\quad\quad s=0,\dots,S,
\end{array}
$$
with $m \geq R$ and $n \geq S$, i.e. not all surface isoparametric curves, corresponding to the knots in $U$ and $V$, are curves of the network. In this case we would have free parameters that could be managed, in order to model the interpolating surface, for example by optimization techniques.

\section*{Acknowledgements} 

The authors thank the University of Torino for its support to their research. Moreover, they are grateful to the anonymous referees for their suggestions and comments that improved the paper.





\end{document}